\documentclass{amsart}
\usepackage{amssymb,amsmath}
\usepackage{tikz}
\usetikzlibrary{arrows,patterns}
\usepackage{amscd,amssymb}

\newtheorem{theorem}{Theorem}[section]
\newtheorem{corollary}[theorem]{Corollary}
\newtheorem{definition}[theorem]{Definition}
\newtheorem{lemma}[theorem]{Lemma}
\newtheorem{proposition}[theorem]{Proposition}
\newtheorem{remark}[theorem]{Remark}

\newtheorem{example}[theorem]{Example}

\def\quiver{Q} 
\def\polytope{\nabla}

\def\ideal{\mathcal{I}}
\def\rep{{\mathrm{Rep}}}
\def\moduli{{\mathcal{M}}}
\def\coord{{\mathcal{O}}}
\def\proj{{\mathrm{Proj}}}
\def\spec{\mbox{{\rm{Spec}} }}
\def\affspan{{\mathrm{AffSpan}}}

\def\supp{\mathrm{supp}}

\def\semigr{S}

\def\uk{{\underline{k}}}
\def\unull{{\underline{0}}}

\def\cocoa{{\hbox{\rm C\kern-.13em o\kern-.07em C\kern-.13em o\kern-.15em A}}}
\def\mc{{\mathbb{C}}}
\def\mz{{\mathbb{Z}}}
\def\mn{{\mathbb{N}_0}}

\def\mr{{\mathbb{R}}}

\def\cone{{\mathrm{Cone}}}

\def\<{\langle}
\def\>{\rangle}

\def\cbgr{K}
\def\cbq{\vec{K}}

\begin{document}
 
\title{Toric quiver cells}
\author{M. Domokos 
\and D. Jo\'o} 
\thanks{Partially supported by  National Research, Development and Innovation Office,  NKFIH K 119934 and PD 121410.  The paper is partially based on research reported in \cite{joo_phd}.}



\date{}
\address{MTA Alfr\'ed R\'enyi Institute of Mathematics\\
Re\'altanoda u. 13-15, 1053 Budapest, Hungary} 
\email{domokos.matyas@renyi.mta.hu} \quad \email{joo.daniel@renyi.mta.hu }
\maketitle


\begin{abstract} 
It is shown that up to dimension four, the toric ideal of a quiver polytope is generated in degree two, with the only exception of the four-dimensional Birkhoff polytope. 
As a consequence, B{\o}gvad's conjecture holds for quiver polytopes of dimension at most four. In arbitrary dimension, the toric ideal of a compressed polytope is generated in degree two if the polytope has no neighbouring singular vertices. Furthermore, the toric ideal of a compressed polytope with at most one singular vertex has a quadratic Gr\"obner basis. 
\end{abstract}

\section{Introduction}

This paper is a sequel to \cite{domokos-joo}, where it was proved that the toric ideal 
of a quiver polytope (called also flow polytope) is generated in degree three. Although most of our present work  concerns binomial defining ideals of polytopal semigroup algebras, we begin the Introduction by recalling a motivation to study these ideals that comes from representation theory of associative algebras, where the language of {\it quivers} (directed graphs) and their representations plays a central role. 
The space of quiver representations with a fixed dimension vector is endowed with the base change action of a product of general linear groups such that the orbits are in bijective correspondence with the isomorphism classes of representations. The corresponding affine quotient spaces parameterize the semisimple representations of the quiver. 
It is well known that when this affine quotient space is  smooth then it can only be an affine space (see for example Theorem 2.1 in \cite{bocklandt}). More sophisticated moduli spaces were introduced  in \cite{king}, where geometric invariant theory (GIT) was applied to construct quasi-projective quotients that parametrize representations satisfying certain stability conditions. Amongst these GIT quotients one can find many non-trivial smooth examples.

The aim of several recent works was to describe the cases when the moduli spaces resulting from these constructions are smooth in terms of the combinatorial properties of the quiver. In \cite{bocklandt} a characterization of quiver settings (i.e. quivers with a fixed dimension vector) with smooth affine quotients were given via a combinatorial reduction process. In \cite{joo} this result was used to describe the same class in a manner analogous to describing classes of graphs by forbidden minors. Thanks to \cite{adriaenssens-lebruyn} 
these results can be applied to decide smoothness of the GIT moduli spaces of quiver representations mentioned above. 
Furthermore, it was shown in \cite{domokos:gmj}  that  a  connected quiver is Dynkin or extended Dynkin if and only if all moduli spaces of its representations are smooth (see also  \cite{chindris}, \cite{bobinski}, \cite{carroll-chindris} for related work).  A key tool to develop results in this direction is to establish methods of simplifying the structure of the quiver without altering the property of being smooth. In Section \ref{sec:arrowremoval} we show in Proposition~\ref{prop:smooth-arrow-removal}  
that if a moduli space of quiver representations is smooth then the moduli spaces we obtain after removing some arrows but keeping the same weight and dimension vector must also be smooth (or empty).  

The rest of the paper is related to the case when the dimension vector of the quiver is set to be one on every vertex. This implies that the moduli spaces that arise are toric varieties, called  {\it toric quiver varieties}. This special case has been studied in \cite{altmann-hille}, \cite{altmann-straten}, \cite{domokos-joo}, \cite{hille:chemnitz}, \cite{hille:canada}, \cite{hille:laa}. Toric quiver varieties come with a canonical embedding into projective space given by a {\it quiver polyhedron} under standard constructions of toric geometry. Applying a result from \cite{yamaguchi-ogawa-takemura} (which proved a conjecture from \cite{diaconis-eriksson}) it was shown in \cite{domokos-joo} that the toric ideal of a projective toric quiver variety in this particular embedding is always generated by its elements of degree at most three. The aim of our work in Section \ref{sec:quivercells} is to refine this result by listing the quiver polytopes up to a fixed dimension that yield a toric ideal that can not be generated in degree two. 
One of the key tools we use here is a hyperplane subdivision method which was also applied in \cite{haase} to study the toric ideals of $3 \times 3$ transportation polytopes. This method allows us to estimate the generators of the toric ideals of quiver polytopes by finding generators for some small subpolytopes which we call {\it quiver cells}. 
Quiver cells turn out to be quiver polytopes themselves. Refining and developing results from \cite{altmann-straten} and \cite{altmann-nill-schwentner-wiercinska} we pointed out in \cite{domokos-joo} that in each dimension there are only finitely many projective toric quiver varieties (up to isomorphism). However, in each dimension there are infinitely many 
quiver polytopes (up to integral-affine equivalence). In contrast in each dimension 
there are finitely many quiver cells  (up to integral-affine equivalence). 
Building on \cite{domokos-joo} we give an approximate classification of them in 
Theorem~\ref{thm:listcells}.  
This is then used to prove our main result Theorem~\ref{thm:cells} saying that the only quiver polytope up to dimension $4$ whose toric ideal is not generated in degree two is the Birkhoff polytope $B_3$. This is a generalization of Proposition 1 in \cite{haase} where the same statement was verified for  $3 \times 3$ transportation polytopes which are quiver polytopes of a complete bipartite quiver.

Smoothness of moduli spaces of quiver representations discussed earlier and quadratic generation of toric ideals come together in  B{\o}gvad's conjecture (see 
Conjecture 13.19 in \cite{sturmfels} or \cite{bruns}) which asserts that the toric ideal of a smooth normal lattice polytope can always be generated in degree two. 
Theorem~\ref{thm:cells} implies that B{\o}gvad's conjecture holds for quiver polytopes of dimension at most four, see Corollary~\ref{cor:bogvad}.  The quiver cells playing essential role in our approach are instances of {\it compressed polytopes}, which are just lattice polytopes of facet width one. Compressed polytopes and their toric ideals have received considerable attention in recent literature and play an important role in linear programming (see for example \cite{haase-etal}, \cite{hibi-ohsugi}, \cite{sullivant} or Chapter 9 of \cite{lorea-rambau-santos}). In Section \ref{sec:relations} we prove quadratic generation for toric ideals of
compressed polytopes under some conditions on the arrangements of the singular points. More precisely, 
Theorems  \ref{thm:0-1deg2} and \ref{thm:grobner}  assert that when a compressed polytope has no neighbouring singular vertices then its toric ideal is always generated in degree two, moreover that if it has at most one singular vertex then its toric ideal possesses a quadratic Gr{\"o}bner basis. These results also provide a basis for further research on proving quadratic generation (in particular, B{\o}gvad's conjecture) for the toric ideals of certain classes of non-compressed polytopes by a similar strategy to that of our Section \ref{sec:quivercells}.

\section*{Acknowledgements} We are indebted to Lutz Hille and Endre Szab{\'o} for their insightful comments on our work. We thank also the referees of the first version of this paper for suggestions improving the manuscript. 

\section{Preliminaries}
\subsection{Moduli Spaces of Quiver Representations}
Let $\quiver$ be a quiver (i.e. a finite directed graph) with vertex set $\quiver_0$ and arrow set $\quiver_1$ (loops, multiple arrows are allowed). 
For an arrow $a\in\quiver_1$ write $a^-$ for the starting vertex and $a^+$ for the terminating vertex of $a$.  
A {\it representation} $R$ of $\quiver$ assigns to each $v\in \quiver_0$ a finite dimensional $\mc$-vector space $R(v)$, and to each $a\in\quiver_1$ a linear map $R(a):R(a^-)\to R(a^+)$. A {\it morphism} between representations $R$ and $S$ is a collection of linear maps  $L(v):R(v)\to S(v)$ $(v\in \quiver_0)$ with $L(a^+)\circ R(a)=S(a)\circ L(a^-)$ for all $a\in\quiver_1$;  $L$ is an {\it isomorphism} if $L(v)$ is a linear isomorphism for all $v\in\quiver_0$. 
The {\it dimension vector} of $R$ is $\dim_{\mc}(R(v)\mid v\in\quiver_0)$. 
 For a fixed dimension vector $\alpha\in\mn^{\quiver_0}$ consider 
\[\rep(\quiver,\alpha):=\bigoplus_{a\in \quiver_1}\hom_{\mc}(\mc^{\alpha(a^-)},\mc^{\alpha(a^+)}).\] 
This is called the {\it space of $\alpha$-dimensional representations} of $\quiver$, as to $x\in \rep(\quiver,\alpha)$ one associates the representation $R_x$ where 
$R_x(v):=\mc^{\alpha(v)}$ is the space of column vectors and $R_x(a):=x(a)$ for $a\in\quiver_1$. 
Clearly for any $\alpha$-dimensional representation of $\quiver$ there exists an $x\in\rep(\quiver,\alpha)$ with $R\cong R_x$. The representations $R_x$ and $R_y$ are isomorphic if and only if $x,y\in\rep(\quiver,\alpha)$ belong 
to the same orbit of the product of general linear groups 
$GL(\alpha):=\prod_{v\in\quiver_0}GL_{\alpha(v)}(\mc)$ acting linearly on $\rep(\quiver,\alpha)$ via 
\[g\cdot R:=(g(a^+)R(a)g(a^-)^{-1}\mid a\in \quiver_1)\quad (g\in GL(\alpha), R\in\rep(\quiver,\alpha)).\] 

A polynomial function $f$ on $\rep(\quiver,\alpha)$ is a {\it relative invariant of weight} $\theta\in\mz^{\quiver_0}$ if 
$f(g\cdot R)=(\prod_{v\in \quiver_0}\det(g(v))^{\theta(v)})f(R)$ holds for all $g\in GL(\alpha)$ and $R\in\rep(\quiver,\alpha)$. 
The relative invariants of weight $\theta$ constitute a subspace $\coord(\rep(\quiver,\alpha))_{\theta}$ in the coordinate ring $\coord(\rep(\quiver,\alpha))$ of the affine space 
$\rep(\quiver,\alpha)$.  Consider  the graded algebra $\bigoplus_{n=0}^{\infty}\coord(\rep(\quiver,\alpha))_{n\theta}$.  Note that the degree zero part of this graded algebra is the algebra $\coord(\rep(\quiver,\alpha))^{GL(\alpha)}$ of polynomial $GL(\alpha)$-invariants on $\rep(\quiver,\alpha)$, whose generators are described in \cite{lebruyn-procesi}.  
In \cite{king} a  quasiprojective variety $\moduli(\quiver,\alpha,\theta)$ is defined as the {\it projective spectrum}  
\[\moduli(\quiver,\alpha,\theta)=\proj(\bigoplus_{n=0}^{\infty}\coord(\rep(\quiver,\alpha))_{n\theta}).\]  
This is a {\it coarse moduli space} for families of $\theta$-semistable $\alpha$-dimensional representations of $\quiver$ up to S-equivalence in the sense of \cite{newstead}. 
Recall that a representation $R$ of $\quiver$ is called {\it $\theta$-semistable} (resp. {\it $\theta$-stable}) if $\sum_{v\in\quiver_0}\theta(v)\dim_{\mc}(R(v))=0$ and  $\sum_{v\in\quiver_0}\theta(v)\dim_{\mc}(S(v))\ge 0$ (resp. $>0$) for all subrepresentations $S$ of $R$. The points $x$ with $R_x$ $\theta$-semistable constitute a Zariski open (possibly empty) subset 
$\rep(\quiver,\alpha)^{\theta-ss}$ in $\rep(\quiver,\alpha)$. There is a map 
$\pi:\rep(\quiver,\alpha)^{\theta-ss}\to\moduli(\quiver,\alpha,\theta)$ such that $(\moduli(\quiver,\alpha,\theta),\pi)$ is a {\it good quotient} of $\rep(\quiver,\alpha)^{\theta-ss}$ by 
$GL(\alpha)$ (cf. Theorem 3.21 in \cite{newstead}). 

\subsection{Toric Varieties}

We refer to \cite{cox-little-schenck} for basic material on toric varieties. 
By a {\it lattice polytope} we mean the convex hull in $\mr^d$ of a finite subset of the lattice $\mz^d\subset \mr^d$. A lattice polytope $\polytope$ is called {\it normal} if for $k \geq 1$ any lattice point in $k\polytope \cap \mz^d$ can be obtained as a sum of $k$ lattice points from $\polytope \cap \mz^d$.
To a positive dimensional lattice polytope $\polytope \subset \mr^d$, we assign the graded submonoid 
$$\semigr(\polytope) = \{(m,\lambda) \in \mz^d\times\mz |\: m \in \lambda\polytope \}\subset \mz^{d+1}$$ 
with degree $k$ part $\semigr(\polytope)_k = (k\polytope\times\{k\})\cap\mz^{d+1}$. For $s_1,s_2 \in \semigr(\polytope)$ we will write $s_1 \geq s_2$ whenever $s_1$ divides $s_2$ in $\semigr(\polytope)$. By slight abuse of the notation we will identify the sets $\polytope \cap \mz^d$ and $\semigr(\polytope)_1$, and for $m \in \polytope\cap \mz^d$ and $s \in \semigr(\polytope)$ write $m \leq s$ whenever $(m,1)$ divides $s$ in $\semigr(\polytope)$. We will denote by $x^s$ the element of the semigroup algebra $\mc[\semigr(\polytope)]$ corresponding to $s \in \semigr(\polytope)$. 

We define the projective toric variety $X_\polytope$ as $\proj\:\mc[\semigr(\polytope)]$. When $\polytope$ is a normal lattice polytope with lattice points $\polytope \cap \mz^d= \{m_1,\dots,m_k\}$ the ring $\mc[\semigr(\polytope)]$ is generated by $\{x^{m_1},\dots,x^{m_k}\}$, and the closure of the image of the map  
$$(\mc^{*})^d\to \mathbb{P}^{k-1},\quad t \mapsto (t^{m_1}:\dots:t^{m_k})$$ 
is a projectively normal embedding of the variety $X_\polytope$.

By the product of two polytopes $\polytope_1 \subset \mr^{d_1}$ and $\polytope_2 \subset \mr^{d_2}$ we mean the polytope $$\polytope_1\times\polytope_2 = \{(x_1,x_2)\in\mr^{d_1+d_2}|\:x_1\in\polytope_1,\:x_2\in\polytope_2\}.$$ The variety $X_{\polytope_1\times\polytope_2}$ is isomorphic to the product of the varieties $X_{\polytope_1}$ and  $X_{\polytope_2}$.

 The homogeneous vanishing ideal of  $X_\polytope$ in this particular embedding is called the {\it toric ideal} of $\polytope$. More explicitly consider the surjection $\varphi$ from the $k$ variable polynomial ring $R = \mc[t_1, \dots,t_k]$  onto $\mc[\semigr(\polytope)]$ defined as $\varphi(t_i) = x^{m_i}$. The toric ideal of $\polytope$ is just  the ideal $\ker(\varphi)$, which we consider along with the standard grading inherited from the polynomial ring $R$. We say that an ideal $I$ of $R$ is {\it generated in degree two} if $I$ is generated by elements of degree at most two (or $I=0$). 
It is well known that
\begin{equation}\label{eq:ker(phi)}
\ker(\varphi)=\mathrm{Span}_{\mc}\{t^a-t^b\mid \sum_{i=1}^k a_im_i=\sum_{j=1}^k b_jm_j\in\semigr(\polytope)\}
\end{equation} 
where for $a=(a_1,\dots,a_k)\in\mn^k$ we write $t^a=t_1^{a_1}\dots t_k^{a_k}$. The equality $\sum_{i=1}^k a_im_i=\sum_{j=1}^k b_jm_j$ forces $\sum_{i=1}^k a_i=\sum_{j=1}^kb_j$. 
In particular $\ker(\varphi)$ is a homogeneous binomial ideal. The toric ideal of $\polytope$ will be denoted by $\ideal(\polytope)$.

We recall a tool from Section 8 of \cite{domokos-joo} for finding the minimal degree of a generating set for $\ideal(\polytope)$. For an element $s \in \semigr(\polytope)$ we define a relation $\sim^*_s$ on the set $\{m \in \semigr(\polytope)_1 \mid m \leq s\}$, as $m_1 \sim^*_s m_2$ if and only if $m_1 = m_2$ or $m_1 + m_2 \leq s$. We will denote by $\sim_s$ the transitive closure of $\sim^*_s$, i.e. $m_1 \sim_s m_2$ if and only if there is a sequence $u_1,\dots,u_k \in \{m \in \semigr(\polytope)_1 \mid m \leq s\}$ such that $m_1 = u_1$, $m_2 = u_k$ and $u_i+u_{i+1} \leq s$ for all $1 \leq i \leq k-1$. Clearly $\sim_s$ is an equivalence relation. Applying Lemma 8.2 and Corollary 8.3 from \cite{domokos-joo} to our setting we have the following:

\begin{proposition}\label{prop:equivclasses}
Let $\polytope$ be a normal lattice polytope with lattice points $\semigr(\polytope)_1= \{m_1,\dots,m_k\}$. For a homogeneous binomial $t^a-t^b \in \ideal(\polytope)$ of degree $r$, with $s = \sum_{i=1}^k a_im_i=\sum_{j=1}^k b_jm_j$, we have that $t^a-t^b$ is contained in the ideal generated by elements of $\ideal(\polytope)$ of degree at most $r-1$ if and only if $m_i \sim_s m_j$ for some $i,j$ with $a_i \neq 0$ and $b_j \neq 0$. In particular the ideal $\ideal(\polytope)$ is generated by its elements of degree at most $r$ if and only if for every $r' > r$ and $s \in \semigr(\polytope)_{r'}$ the relation $\sim_s$ has precisely one equivalence class.
\end{proposition}

For a set $H \subseteq \mr^d$ we will denote by $\cone(H)$ the set $\{\lambda x|\:x\in H,\:\lambda\in\mr^+\}$. We say that a convex cone is {\it rational} if it equals $\cone(H)$ for a finite set of integer vectors $H$. By a {\it ray generator} of a rational cone we mean the lattice point closest to the origin on one of its edges.
Let $\polytope$ be a normal lattice polytope. For a vertex $v \in \polytope$ we denote by $U_v$ the principal affine subset of $X_\polytope$ defined by $x^v \neq 0$. It can be deduced from the definition of $X_\polytope$ that $U_v$ is isomorphic to the spectrum of the semigroup algebra $\mc[Cone(\polytope-v)\cap\mz^d])$. Moreover it follows from elementary facts of toric geometry that $U_v$ is smooth if and only if it is an affine space which in turn happens precisely when the ray generators of $Cone(\polytope-v)$ are part of a $\mz$-basis of the lattice $\mz^d$. We will say that $v$ is a {\it smooth vertex} whenever $U_v$ is smooth, and that it is a {\it singular vertex} otherwise.


\subsection{Toric quiver varieties}

We briefly recall the definition and some elementary properties of toric quiver varieties and refer to \cite{domokos-joo} and the references there for further details. 

Let $\quiver$ be a quiver without oriented cycles and set $\alpha = (1,\dots,1)$. The product of general linear groups acting on $\rep(\quiver,\alpha)$ is then just a $|\quiver_0|$ dimensional torus and the moduli spaces $\moduli(\quiver,\alpha,\theta)$ are projective toric varieties, which we call {\it toric quiver varieties}. We will simply write $\moduli(\quiver,\theta)$ to denote the variety $\moduli(\quiver,(1,\dots,1),\theta)$. We define the {\it quiver polytope} 
(called {\it flow polytope} in \cite{schrijver}) 
corresponding to the pair $(\quiver, \theta)$ as, 
\begin{equation}\label{eq:quiverpolytope}\polytope(\quiver,\theta)=\{x\in\mr^{\quiver_1}_{\geq 0}\mid \forall v\in\quiver_0:\quad \theta(v)=\sum_{a^+=v}x(a)-\sum_{a^-=v}x(a)\}.
\end{equation}
\
The polytope $\polytope(\quiver,\theta)$ is a normal lattice polytope (see Theorem 13.14 in \cite{schrijver}). When $\moduli(\quiver,\theta)$ is nonempty it is isomorphic to $X_{\polytope(\quiver,\theta)}$ (see Proposition 3.1 of \cite{domokos-joo}).
While in the current work we only focus on projective toric quiver varieties, we note that the above statements can be generalized to quivers containing oriented cycles, in which case $\polytope(\quiver,\theta)$ is a lattice polyhedron and the variety $\moduli(\quiver,(1,\dots,1),\theta)$ is quasi-projective. 

The class of smooth quiver moduli spaces was characterized  in \cite{bocklandt} by an algorithmic method. A different description was given in \cite{joo} using certain forbidden configurations. For the purposes of the current paper we only need the following proposition regarding smooth vertices, which is a direct consequence of Theorem 6.2 in \cite{domokos-joo}. For $m \in \polytope(\quiver,\theta)$ we denote by $\supp(m)$ the set $\{a\in \quiver_1|\:m(a)\neq 0\}$.
\begin{proposition}\label{prop:smoothvertex}
Let $v$ be a vertex of $\polytope(\quiver,\theta)$. If the quiver with vertex set $\quiver_0$ and arrow set $\supp(v)$ is connected (in the undirected sense) then $v$ is a smooth vertex.
\end{proposition}

We will write $\ideal(\quiver,\theta)$ for the ideal $\ideal(\polytope(\quiver,\theta))$ and $\semigr(\quiver,\theta)$ for the monoid $\semigr(\polytope(\quiver,\theta))$. We recall from Theorem 9.3 of \cite{domokos-joo} (the same result was announced in \cite{lenz}, but the proof was withdrawn later in \cite{lenz-withdrawn}):

\begin{theorem}\label{thm:deg3}
Let $\quiver$ be a quiver without oriented cycles, and $\theta$ an integer weight such that $\polytope(\quiver,\theta)$ is non-empty. Then the ideal $\ideal(\quiver,\theta)$ is generated by its elements of degree at most $3$. 
\end{theorem}

We will denote by $\cbgr_{n,k}$ the complete bipartite graph with $n$ vertices on one side of the bipartition and $k$ on the other, and by $\cbq_{n,k}$ the quiver we obtain from $\cbgr_{n,k}$ by orienting each edge so the side with $n$ vertices consists of sources and the side with $k$ vertices consists of sinks. 
Let $\theta$ be the weight on the vertices of $\cbq_{n,n}$ that takes value $1$ on each sink and $-1$ on each source. The quiver polytope $\polytope(\cbq_{n,n},\theta)$ is then isomorphic to the Birkhoff-polytope $B_n$, which is usually defined as the convex hull of $n\times n$ permutation matrices. When $n = 3$, the polytope $B_3$ has $6$ vertices which can be indexed by the permutations $(i,j,k) \in S_3$. One can verify without difficulty that in this case the toric ideal $\ideal(B_3)$ is generated by the degree $3$ binomial 
$$t_{1,2,3}t_{2,3,1}t_{3,1,2} - t_{2,1,3}t_{3,2,1}t_{1,3,2} .$$


\section{Smoothness is preserved by arrow removal}\label{sec:arrowremoval}

An injective morphism of algebras $\iota: R\hookrightarrow S$ is called an {\it algebra retract} if there is a surjective morphism $\varphi:S \twoheadrightarrow R$ such that $\varphi \circ \iota = id_R$. The morphism $\varphi$ is called the {\it retraction map}. When $R$ and $S$ are both graded and the morphisms $\iota,\varphi$ are graded morphisms we call  $\iota: R\hookrightarrow S$ a {\it graded algebra retract}.\par\smallskip

By a complete intersection we shall mean an ideal theoretic complete intersection. Recall that for the graded algebra $S$ the points of the scheme $\proj\:S$ are the homogeneous prime ideals in $S$ that do not contain the irrelevant ideal $S_+$, and the stalk at the point $p \in \proj\:S$ is the homogeneous localization $S_{(p)}$ defined as the subring of degree zero elements in the localized ring $T^{-1}S$, where $T$ consists of the homogeneous elements that are not in $p$. We shall write $S_p$ for the ordinary localization of $S$ at the prime ideal $p$. 

\begin{proposition}\label{prop:proj-retract}
Let $\iota: R\hookrightarrow S$ be a graded algebra retract. Then if the variety $\proj\:S$ is smooth (resp. locally a complete intersection) then $\proj\:R$ is also smooth  (resp. locally a complete intersection).
\end{proposition}
 
\begin{proof}
Let $\varphi$ denote the retraction morphism $S \twoheadrightarrow R$. We recall from Proposition 2.10 of \cite{epstein-nguyen} that for any prime ideal $p \in \spec R$ and  $q = \varphi^{-1}(p) \in \spec S$ we have a natural algebra retract of the localized rings $R_p \hookrightarrow S_q$. Since $\iota,\varphi$ preserve the grading if $p$ is a homogenous prime ideal of $R$ then $q = \varphi^{-1}(p)$ is a homogenous prime ideal of $S$ and there is a natural algebra retract of the homogeneous localizations $R_{(p)} \hookrightarrow S_{(q)}$, moreover if $p$ does not contain the irrelevant ideal of $R$ then $q$ does not contain the irrelevant ideal of $S$. Now the proposition follows from Theorem 3.2 of \cite{epstein-nguyen} which asserts that every algebra retract of a regular (resp. locally complete intersection) ring is also a regular (resp. locally complete intersection) ring.
\end{proof}

\begin{proposition}\label{prop:smooth-arrow-removal} 
Let $\quiver'$ be a quiver obtained from $\quiver $ by removing an arrow 
$a\in \quiver_1$. 
\begin{itemize}
\item[(i)] If $\moduli(\quiver,\alpha,\theta)$ is smooth, then $\moduli(\quiver',\alpha,\theta)$ is also smooth (or empty).
\item[(ii)] If $\moduli(\quiver,\alpha,\theta)$ is locally a complete intersection, then $\moduli(\quiver',\alpha,\theta)$ is also locally a complete intersection (or empty). 
\end{itemize}  
\end{proposition} 

\begin{proof} 
View $\rep(\quiver',\alpha)$ as a direct summand of $\rep(\quiver,\alpha)$ in the obvious way, and denote by $\iota$ the algebra retract $\coord(\rep(\quiver',\alpha))\hookrightarrow\coord(\rep(\quiver,\alpha))$ induced by the projection $\rep(\quiver,\alpha)\to\rep(\quiver',\alpha)$, and by $\varphi$ the corresponding retraction map induced by the embedding $\rep(\quiver',\alpha)\hookrightarrow \rep(\quiver,\alpha)$. Since both $\iota$ and $\varphi$ are $GL(\alpha)$ equivariant  it follows that $\iota(\coord(\rep(\quiver',\alpha))_{n\theta}) \subseteq \coord(\rep(\quiver,\alpha))_{n\theta}$ and $\varphi(\coord(\rep(\quiver,\alpha))_{n\theta}) \subseteq \coord(\rep(\quiver',\alpha))_{n\theta}$, and hence we have the graded algebra retract $$\bigoplus_{n=0}^{\infty}\coord(\rep(\quiver',\alpha))_{n\theta} \hookrightarrow \bigoplus_{n=0}^{\infty}\coord(\rep(\quiver,\alpha))_{n\theta}.$$ 
Now (i) and (ii) follow from Proposition \ref{prop:proj-retract}. 
\end{proof} 


\section{Quiver polytopes with degree 3 relations up to dimension 4}\label{sec:quivercells}

For the rest of the paper we assume that $\quiver$ has no oriented cycles. 
The aim of this section is to compile a full list of quiver polytopes whose toric ideals are not generated in degree two up to dimension $4$. 
First we need to clarify when we consider two polytopes to be the same:

\begin{definition}\label{def:isomorphicpolytopes}
{\rm The lattice polytopes  $\polytope_i\subset V_i$ with lattice $M_i\subset V_i$ $(i=1,2)$ are {\it integral-affinely equivalent} 
if there exists an affine linear isomorphism  $\varphi:\affspan(\polytope_1)\to\affspan(\polytope_2)$ of affine subspaces 
with the following properties: 
\begin{itemize}
\item[(i)] $\varphi$ maps $\affspan(\polytope_1)\cap M_1$ onto $\affspan(\polytope_2)\cap M_2$; 
\item[(ii)] $\varphi$ maps  $\polytope_1$ \ onto $\polytope_2$. 
\end{itemize} }
\end{definition} 

It is easily verified that integral-affinely equivalent polytopes have the same toric ideals. Next we recall a reformulation of some of the results in Section 4 of \cite{domokos-joo}, that allow us to classify quiver polytopes in a given dimension. By the {\it underlying graph} of a quiver $\quiver$, we mean the undirected graph we obtain by forgetting the orientation of the arrows of $\quiver$. An undirected graph is  said to be {\it prime}  if it has at least one edge and is biconnected (i.e. it is not the union of two proper subgraphs having at most one common vertex). 
A quiver is called {\it prime} if its underlying graph is prime. It is well known that any undirected graph decomposes into a tree of maximal prime subgraphs, yielding us a decomposition $\quiver^1 \cup \dots \cup \quiver^k$ of the quiver $\quiver$ as the union of maximal prime subquivers. Note that by maximality no two of the $Q^i$ can have more than one vertex in common. Proposition \ref{prop:quiverproducts} below shows, that to obtain a full list of quiver polytopes in a given dimension such that their ideals are not generated in degree two, one only has to consider prime quivers and products of lower dimensional examples. 
In Lemma \ref{lem:productpoly} we recall a well known fact about toric ideals (see \cite[Corollary 2.10]{sullivant2}  or \cite[Proposition 2.7]{ohsugi-hibi_2010} for a more general statement). A short proof is provided using our formalism for the convenience of ther reader.

\begin{lemma}\label{lem:productpoly}
Let $\polytope_1 \subset \mr^{d_1}$ and $\polytope_2 \subset \mr^{d_2}$ be normal lattice polytopes and $\polytope = \polytope_1\times\polytope_2 \subset \mr^{d_1 + d_2}$. Denote the corresponding toric ideals by $\ideal(\polytope)$, $\ideal(\polytope_1)$ and $\ideal(\polytope_2)$. Then $\ideal(\polytope)$ is generated in degree two if and only if both $\ideal(\polytope_1)$ and $\ideal(\polytope_2)$ are generated in degree two.
\end{lemma}
\begin{proof}
Recall that $S(\polytope)_k = (k\polytope\times\{k\})\cap\mz^{d+1}$ is the degree $k$ part of the graded monoid $S(\polytope)$, and that for $s_1,s_2 \in S(\polytope)$ we write $s_2 \leq s_1$ whenever $s_2$ divides $s_1$ in $S(\polytope)$.

First assume that $\ideal(\polytope)$ is generated in degree two and pick any $s_1 \in S(\polytope_1)_k$ for $k \geq 3$. By Proposition \ref{prop:equivclasses} we need to show that $\sim_{s_1}$ has only one equivalence class.  Choose any $m_1,m_2 \in S(\polytope_1)_1$ such that $m_1,m_2 \leq s_1$,  moreover choose an $s_2 \in  S(\polytope_2)_k$ and $n \in S(\polytope_2)_1$ such that $n \leq s_2$. By the assumption we have $(m_1,n) \sim_{(s_1,s_2)} (m_2,n)$, hence there is a sequence as in the definition of $\sim_s$ starting from $(m_1,n)$ and ending in $(m_2,n)$. Projecting this sequence to the coordinates that correspond to $\polytope_1$ we obtain that $m_1 \sim_{s_1} m_2$.\par

For the other direction let $(s_1,s_2) \in S(\polytope)_k$ ($k \geq 3$) and $(m_1,n_1), (m_2,n_2) \in S(\polytope)_1$ such that $(m_1,n_1),(m_2,n_2) \leq (s_1,s_2)$.  By the assumption that $\ideal(\polytope_1)$ is generated in degree two there is a sequence $w_1,\dots,w_j \in S(\polytope)_1$ such that $w_1 = m_1$, $w_j = m_2$ and $w_i+w_{i+1} \leq s_1$ for all $i = 1,\dots,j-1$. Note that we can always assume $j \geq 3$.  By normality of $\polytope_2$ we have $s_2 = n_1 + u_1 + \dots u_{k-1}$ for some $u_1,\dots,u_{k-1} \in S(\polytope_2)_1$. Since $k,j \geq 3$ it is possible to choose a sequence $p_2,\dots,p_{j-1}$ such that $p_i \in \{n_1,u_1,\dots,u_{k-1}\}$, $p_i \neq p_{i+1}$ and $n_1 \notin \{p_2, p_{j-1}\}$. Now the sequence $(m_1, n_1), (w_2, p_2),\dots,(w_{j-1},p_{j-1}),(m_2,n_1)$ satisfies the conditions in the definition of $\sim_{(s_1,s_2)}$ 
so  $(m_1, n_1) \sim_{(s_1,s_2)} (m_2,n_1)$. Applying the same argument for $n_1$ instead of $m_1$ we have  $(m_1, n_1) \sim_{(s_1,s_2)} (m_2,n_2)$ completing the proof.
\end{proof}

\begin{proposition}\label{prop:quiverproducts}
Let $\quiver$ be a quiver with prime components $\quiver^1, \dots, \quiver^k$, and $\theta$ an integer weight on the nodes of $\quiver$. Then there exist weights $\theta_1, \dots, \theta_k$ on the nodes of  $\quiver^1, \dots, \quiver^k$ respectively, such that $\polytope(\quiver,\theta)$ is integral-affinely equivalent to $\prod_{i=1}^k \polytope(\quiver^i, \theta_i)$. Moreover $\ideal(\quiver, \theta)$ is generated by its elements of degree at most $2$ if and only if $\ideal(\quiver^i, \theta_i)$ is generated by its elements of degree at most $2$ for all $i = 1, \dots, k$.
\end{proposition}
\begin{proof}
The first statement follows directly from Proposition 4.10 of \cite{domokos-joo}. The second follows from Lemma \ref{lem:productpoly}.
\end{proof}

Next we recall a reformulation of the results in Section 4 of \cite{domokos-joo}, that allow us to classify quiver polytopes that occur in a fixed dimension. By the {\it valency} of a vertex of a quiver (resp. undirected graph)  we mean the number of arrows (resp. edges) incident to it. A $3$-regular graph is just a graph in which every vertex has valency $3$. We remind that we allow quivers (and graphs) to have multiple arrows (resp. edges). For an undirected graph $G$ we will denote by $G^*$ the quiver we obtain by placing a valency $2$ sink on each edge of $G$, as illustrated on the figure below.
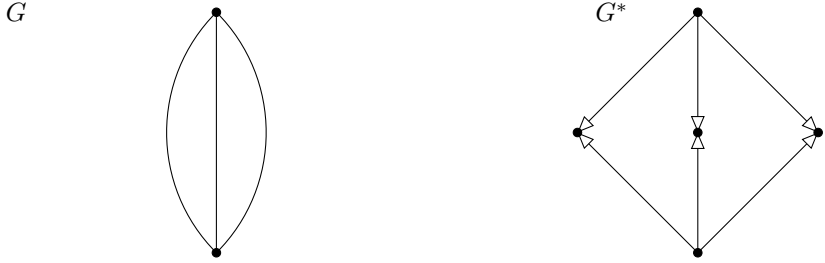
\[
\begin{tikzpicture}[>=open triangle 45,scale=0.8] 

\foreach \x in {(0,0),(0,4)} \filldraw \x circle (2pt);
\draw  [-]  (0,0)--(0,4);
\draw  [-]  (0,0) to [out=45,in=-45] (0,4);
\draw  [-]  (0,0) to [out=135,in=-135] (0,4);

\node[left] at (-3,4) {$G$};

\foreach \x in {(8,0),(8,2),(6,2),(10,2),(8,4)} \filldraw \x circle (2pt);

\node[left] at (7,4) {$G^*$};
\draw  [->]  (8,0)--(8,2);
\draw  [->]  (8,4)--(8,2);
\draw  [->]  (8,0)--(6,2);
\draw  [->]  (8,4)--(6,2);
\draw  [->]  (8,0)--(10,2);
\draw  [->]  (8,4)--(10,2);

\end{tikzpicture}
\]
 By a {\it prime} lattice polytope we mean a lattice polytope that is not the product of lattice polytopes of strictly smaller dimensions.

\begin{theorem}\label{thm:allquivers}
Let $\polytope(\quiver,\theta)$ be a prime quiver polytope of dimension $d \geq 2$. Then there exist a prime $3$-regular graph $G$, without loops, on $2d-2$ nodes and an integer weight $\theta^*$ on the nodes of $G^*$ such that $\polytope(\quiver,\theta)$ is integral-affinely equivalent to $\polytope(G^*,\theta^*)$.
\end{theorem}

In \cite{domokos-joo} we applied Theorem \ref{thm:allquivers} to show that it is possible to list every toric quiver variety in a given dimension. However we estimate these lists in dimension $4$ and higher to be extremely long. More importantly to each projective toric quiver variety there are an infinite number of quiver polytopes associated,  hence there did not seem to be any obvious way to achieve our goal via direct computation. Instead we follow an approach similar to that of \cite{haase}, where it was shown that amongst the $3\times3$ transportation polytopes, only the Birkhoff polytope $B_3$ yields a toric ideal which is not generated in degree two. This is a special case of our result, since $3\times 3$ transportation polytopes are quiver polytopes of the bipartite quiver $\cbq_{3,3}$. The key tool in their proof was to use a hyperplane subdivision to decompose the polytopes into "cells", which are subpolytopes of facet width $1$, and then carry out a case by case analysis of the - finitely many - cells that occur. \par\smallskip
For a lattice polytope $\polytope \subset \mr^d$ and an integer vector $\uk \in \mz^d$ we define the {\it $\uk$-cell} of $\polytope$ to be $$\polytope_{\uk} = \{ x\in \polytope \mid \uk(i) \leq x(i) \leq \uk(i)+1 \:\forall( i: 1 \leq i \leq d) \}.  $$

By the $\unull$-cell of $\polytope$ we just mean the $k$-cell for $k = (0,\dots,0)$. For a quiver polytope $\polytope(\quiver,\theta)$ and a non-negative integer vector $k \in \mn^{\quiver_1}$  we will denote by $\theta_{\uk}$ the weight defined by $$\theta_{\uk} (v) =  \sum_{a^+=v}\uk (a)-\sum_{a^-=v}\uk (a),$$
i.e. $\theta_{\uk}$ is the unique weight $\theta'$ such that $\uk \in \polytope(\quiver,\theta')$. \par\smallskip
We will call the non-empty cells that can be obtained from quiver polytopes {\it quiver cells}. Since quiver polytopes always lie in the positive quadrant it suffices to consider cells defined by $\uk \in \mn^{\quiver_1}$. As we will show in the next proposition, one only has to consider a finite set of weights to obtain a complete list of quiver cells (up to translation) associated to a fixed quiver.
\begin{lemma}\label{lemma:finitecells}
Let $\quiver$ be a quiver and $\uk \in \mn^{\quiver_1}$. 
\begin{itemize}
\item[(i)] For any integer weight $\theta$ we have that  $\polytope(\quiver,\theta)_{\uk} = \polytope(\quiver,\theta - \theta_\uk)_\unull + \uk$.
\item[(ii)] There are only finitely many different weights $\theta$, such that $\polytope(\quiver,\theta)_\unull$ is non-empty.
\end{itemize} 
\end{lemma}

\begin{proof}
(i) follows immediately from the definition of quiver polytopes. For (ii) note that the lattice points of $\polytope(\quiver,\theta)_\unull$ take values $\{0,1\}$ on each edge, hence if the polytope is non-empty, $\theta$ satisfies $$-|\{a\in \quiver_1 \mid a^- = v\}| \leq \theta(v) \leq |\{a\in \quiver_1 \mid a^+ = v\}|, $$ for each vertex $v \in \quiver_0$.
\end{proof}

It follows from Proposition 2.2 of \cite{domokos-joo} that quiver cells are also quiver polytopes (see Theorem~\ref{thm:listcells} (ii) below for a more precise statement), in particular they are normal lattice polytopes and their ideals of relations are generated in degree at most $3$.  \par\smallskip

Denote by $M^{\quiver}$ the lattice consisting of integer vectors $m \in \mz^{\quiver_1}$  satisfying
\begin{equation}\label{eq:cycle} 
\sum_{a^+=v}m(a)-\sum_{a^-=v}m(a) = 0,
\end{equation}
for all $v \in \quiver_0$.  By an {\it alternating cycle} of the quiver we mean an element $c \in M^{\quiver}$ with entries in the set $\{0,1,-1\}$, satisfying that $\mbox{supp}(c)$ is a primitive cycle of the underlying graph of $\quiver$ (i.e. a cycle that does not run through the same vertex twice). It is not hard to show that the alternating cycles generate the lattice $M^{\quiver}$. Moreover a simple inductive argument shows that any lattice point $m \in M^{\quiver}$ can be greedily decomposed as a sum of alternating cycles $c_1,\dots,c_l$, such that the coordinates of the $c_i$ are either zero or have the same sign as the corresponding coordinate of $m$. Alternatively one can derive this statement from Theorem 21.2 from \cite{schrijver:linear} along with the discussion that follows it and the fact that vertex-arrow incidence matrices of quivers are totally unimodular. 
Lemma~\ref{lemma:centering} shows us why cells play an important role in studying the generators of toric ideals. 

\begin{lemma}\label{lemma:centering}
For a quiver polytope $\polytope(\quiver,\theta)$, and a relation $b = t_{m_1}t_{m_2}t_{m_3}-t_{n_1}t_{n_2}t_{n_3} \in \ideal(\quiver,\theta)$ let $\uk \in \mn^{\quiver_1}$ be such that $(m_1 + m_2 + m_3)/3 \in \polytope(\quiver,\theta)_\uk$. Then there exist $m_1',m_2',m_3' \in \polytope(\quiver,\theta)_\uk$ and $n_1',n_2',n_3' \in \polytope(\quiver,\theta)_\uk$ such that $b' =  t_{m'_1}t_{m'_2}t_{m'_3}-t_{n'_1}t_{n'_2}t_{n'_3}\in \ideal(\quiver,\theta)$ and $b$ is contained in an ideal generated by $b'$ and some quadratic elements of $\ideal(\quiver,\theta)$. 
In particular, $\ideal(\quiver,\theta)$ is generated by the toric ideals of the cells 
$\polytope(\quiver,\theta)_{\uk}$ together with some quadratic relations. 
\end{lemma}

\begin{proof}
We need to show that we can transform $m_1,m_2,m_3$ into 
$m_1',m_2',m_3'\in \polytope(\quiver,\theta)_\uk$ by successively replacing a pair of lattice points with a new pair of lattice points that have the same sum. For 
$n\in \mz^{\quiver_1}$ and an arrow $a\in \quiver_1$ denote by $d_a(n)$ the distance of $n(a)$ from the set $\{\uk(a),\uk(a)+1\}$ and set 
$$D(m_i) = \sum_{a\in\quiver_1} d_a(m_i).$$ If $D(m_1) + D(m_2) + D(m_3) = 0$ we are done since $D(m_i) = 0$ if and only if $m_i \in  \polytope(\quiver,\theta)_\uk$.
Otherwise we will show that we can replace two of $m_1,m_2,m_3$ with new lattice points from $\polytope(\quiver,\theta)$ having the same sum, such that the value of $D(m_1) + D(m_2) + D(m_3)$ strictly decreases, and  - since the $D(m_i)$ are all non-negative integers - in finitely many steps the sum will be $0$.  So suppose that $D(m_1) + D(m_2) + D(m_3) > 0$. Then one of the $m_j$, say $m_1$ does not lie in the cell $\polytope(\quiver,\theta)_\uk$. Thus we may assume that say $m_1(a) < \uk(a)$ for some $a \in \quiver_1$ 
(the case $m_1(a)>\uk(a)+1$ can be dealt with similarly). Since $\uk(a) \leq (m_1(a) + m_2(a) + m_3(a))/3$ it follows that one of $m_2(a)$ and $m_3(a)$ has to be at least $k(a)+1$. Assume that $m_2(a) \geq k(a)+1$. Now since $m_2-m_1 \in M^{\quiver}$ it decomposes as a sum of alternating cycles $c_1 + \dots + c_l$. Since $m_2(a) > m_1(a)$ we have that $c_j(a) = 1$  for some $j$. Recall that $c_1,\dots,c_l$ can be chosen such that their  coordinates are either zero or have the same sign as the corresponding coordinate of $m_2-m_1$. Therefore the lattice points $m_1+c_j, m_2 - c_j$ have non-negative entries, hence   $m_1 + c_j, m_2 - c_j \in \polytope(\quiver,\theta)$. We claim that for any $b \in \quiver_1$ we have that 
\begin{eqnarray}\label{eq:claim}
d_b(m_1) + d_b(m_2)  \geq d_b(m_1+c_j) +d_b(m_2 - c_j) 
\\ \notag \mbox{with strict inequality when }b = a.
\end{eqnarray} 
If $c_j(b) = 0$ then the inequality \eqref{eq:claim} holds  trivially. If $c_j(b) \neq 0$ then $m_1(b) \neq m_2(b)$ and $c_j(b) = \mathrm{sign}(m_2(b)-m_1(b))$. It follows that 
$$|d_b(m_1+c_j) - d_b(m_1)| \leq 1 \quad \mbox{ and }\quad |d_b(m_2-c_j) - d_b(m_2)| \leq 1$$ and an easy case-by-case analysis taking into account the possible relative positions of $m_1(b)$, $m_2(b)$ and $\uk(b)$ shows \eqref{eq:claim}. 
Moreover in the case $b = a$, we have that $c_j(a) = 1$ and $$d_a(m_1) > d_a(m_1+1),$$
since $m_1(a) < \uk(a)$, and $$d_a(m_2) \geq d_a(m_2-1),$$
since $m_2(a) \geq \uk(a)+1$. It follows from \eqref{eq:claim} that $$D(m_1) + D(m_2) + D(m_3) > D(m_1 + c_j) + D(m_2-c_j) + D(m_3)$$ and induction completes the proof.
\end{proof}

\begin{remark} {\rm 
(i) The above argument differs from Section 2 of \cite{haase} which depends on the fact that in the case of $3\times 3$ transportation polytopes the cells that do not yield cubic relations, also possess a quadratic Gr\"obner basis (in fact they are simplices). 

(ii) An alternative proof of Lemma~\ref{lemma:centering} could be derived from Theorem 6.2 in \cite{bruns}, which  is proven using several facts about the defining ideals of monoidal complexes.

(iii) The relation $\sim_s$ as a relation  on the set $\{m\in \semigr(\polytope(\quiver,\theta)_{\uk})_1\mid m\le s\}$ is not the same as the restriction to  $\{m\in \semigr(\polytope(\quiver,\theta)_{\uk})_1\mid m\le s\}$ of the corresponding relation $\sim_s$ on the  set 
$\{m\in \semigr(\quiver,\theta)_1\mid m\le s\}$. }
\end{remark}

Lemma \ref{lemma:finitecells} provides us with a way to calculate a complete list of quiver cells that can be obtained from a fixed quiver. One needs to check for a finite set of weights whether the polytope $\polytope(\quiver,\theta)_\unull$ is non-empty. Our next goal is to show that the list of weights to be considered can be singificantly reduced  if we are only interested in full-dimensional cells. 

 For a quiver $\quiver$ we set $\chi(\quiver) = |\quiver_1|-|\quiver_0|+\chi_0(\quiver)$, where $\chi_0(\quiver)$ denotes the number of connected components of $\quiver$. Note that if $G$ is a connected $3$-regular graph on $2d-2$ vertices then $\chi(G^*) = d$. Set $H$ to be the affine subspace of $\mr^{\quiver_1}$ defined by the set of equations 
 $\{\theta(v)=\sum_{a^+=v}x(a)-\sum_{a^-=v}x(a)\mid v\in\quiver_0\}$. It is well known that $\dim(H) = \chi(\quiver)$ (see for example (13.26) in \cite{schrijver}). Since  $\polytope(\quiver,\theta)$ is just the intersection of $H$ with the positive quadrant it follows that  $\dim(\polytope(\quiver,\theta)) \leq \chi(\quiver)$, for any weight $\theta$.

We first  prove the following Theorem:

\begin{theorem}\label{thm:listcells}
\begin{itemize}
\item[(i)] For a graph $G$ and a non-empty quiver cell $\polytope(G^*, \theta)_\uk$, there exists a  subgraph $H$ of $G$ on the same set of vertices and there exists a weight 
$\theta'\in \mz^{H^*_1}$   that takes value $1$ on every sink of $H^*$, such that $\polytope(G^*, \theta)_\uk$ is integral-affinely equivalent to $\polytope(H^*, \theta')$. Moreover if $G$ is prime and $\dim (\polytope(G^*, \theta)_\uk) = \chi(G^*)$ then $H = G$. 
\item[(ii)] Let $G$ be a prime 3-regular graph on $2d-2$ vertices ($d \geq 2$). Then every $d$-dimensional quiver cell $\polytope(G^*, \theta)_\uk$ is integral-affinely equivalent to  $\polytope(G^*, \theta')$, where $\theta'$ is a weight that  takes value $1$ on every sink of $G^*$, value $-1$ on $d-1$ of the sources in $G^*$ and value $-2$ on the remaining $d-1$  sources.
\end{itemize}
\end{theorem}
\begin{proof}
By (i) of Lemma \ref{lemma:finitecells} it is sufficient to deal with the case 
$\uk=\unull$. 
Let $w$ be a valency $2$ sink of $G^*$ and $v_1,v_2$ the sources neighbouring $w$. 
If $\theta(w) < 0$ or $\theta(w) >2$ then clearly $\polytope(G^*,\theta)_{\unull}$ is empty. 
Let us assume now that $\theta(w) = 2$. It follows  that every $x \in \polytope(G^*,\theta)_{\unull}$ takes value $1$ on both arrows incident to $w$. Let $H$ be the graph we obtain from $G$ by deleting the edge on which $w$ was placed and $\theta'\in \mz^{H^*_1}$ the weight  for which $\theta'(v_1) = \theta(v_1)+1$,  $\theta'(v_2) = \theta(v_2)+1$ and $\theta'  = \theta$ on the rest of the vertices of $H^*$ (identified with the corresponding vertices of $G^*$). By considering the natural projection from $\mr^{G^*_1}$ to $\mr^{H^*_1}$ we see that $\polytope(G^*,\theta)_{\unull}$ is integral-affinely equivalent to $\polytope(H^*,\theta')_{\unull}$.  If $\theta(w) = 0$ then every  $x \in \polytope(G^*,\theta)_{\unull}$ takes value $0$ on both arrows incident to $w$ and it follows by a similar argument that $\polytope(G^*,\theta)_{\unull}$ is integral-affinely equivalent to $\polytope(H^*,\theta')_{\unull}$ where $\theta' = \theta$ on every vertex of $H^*$.
Applying the above steps repeatedly we obtain that $\polytope(G^*,\theta)_{\unull}$ is integral-affinely equivalent to $\polytope(H^*,\theta')_{\unull}$ where $H$ is a spanning subgraph of $G$ and $\theta'$ takes value $1$ on every sink. On the other hand in this case we have $\polytope(H^*,\theta')_{\unull} = \polytope(H^*,\theta')$, 
because every arrow of $H^*$ is incident to a sink, so the condition $\theta'(v)=1$ for each sink $v$ implies $0\le x(a)\le 1$ for all $x\in \polytope(H^*,\theta')$. 
The first statement of (i) is proven. We turn to the proof of the second statement of (i). Note that if $G$ is prime and $H^*$ is a quiver obtained by removing a sink and the arrows adjacent to it from $G^*$ then $|H^*_1|=|G^*_1|-2$, $|H^*_0|=|G^*_0|-1$, 
$\chi_0(H^*)=\chi_0(G^*)$, hence $\chi(H^*)<\chi(G^*)$. On the other hand if 
$\polytope(G^*,\theta)_\uk$ is integral-affinely equivalent to $\polytope(H^*,\theta')$, then 
$\chi(G^*)=\dim \polytope(G^*,\theta)_\uk=\dim \polytope(H^*,\theta')\le \chi(H^*)$, a contradiction.   

For (ii) first note that $\chi(G^*) = d$ so it follows from (i) that every d-dimensional quiver cell $\polytope(G^*, \theta)_\uk$ is integral-affinely equivalent to  $\polytope(G^*, \theta')$, where $\theta'$ is a weight that  takes value $1$ on every sink of $G^*$. Let $v$ be a source of $G^*$.  If $\theta'(v) > 0$ or $\theta'(v) < -3$ then $\polytope(G^*, \theta')$ is empty, since $v$ has valency $3$. Moreover if $\theta'(v) = 0$ (resp. $\theta'(v) = -3$) then every  $x \in \polytope(G^*,\theta')$ takes value $0$ (resp. value $1$) on the arrows incident to $v$ and arguing similarily as in (i) it follows that $\dim(\polytope(G^*, \theta')) < \chi(G^*) = d$. Finally since $\polytope(G^*,\theta')\neq\emptyset$ implies that $\sum_{v \in G^*_0} \theta'(v) = 0$ we conclude that $\theta'$ takes value $-1$ on exactly $d-1$ of the $2d-2$ sources in $G^*$ and value $-2$ on the rest of the sources.
\end{proof}

Furthermore the following Lemma shows us that for our purpose it is enough to deal with the cases when $G$ is a simple graph (i.e. it contains no multiple edges).

\begin{lemma}\label{lem:multiplearrow}
Let $G$ be a graph, containing two edges - $e_1$ and $e_2$ - running between the same vertices and denote by $v_1, v_2$ the valency $2$ sinks of $G^*$ that are placed on $e_1$ and $e_2$ respectively. Let $H$ be the graph we obtain from $G$ by collapsing $e_1$ and $e_2$ into a single edge $e$ and denote by $v$ the valency $2$ sink of $H^*$ that is placed on $e$. Let  $\theta$ be a weight on $G^*$, with $\theta(v_1) = \theta(v_2) =  1$ such that $\polytope(G^*,\theta)$ is non-empty and the toric ideal  $\ideal(G^*,\theta)$ is not generated by its quadratic elements. Then there exists a weight $\theta'$ on $H^*$ such that the toric ideal $\ideal(H^*,\theta')$ is not generated by its quadratic elements.
\end{lemma}
\begin{proof}
Denote the arrows of $G^*$ and $H^*$ incident to $v_1,v_2$ and $v$ as in the picture below.
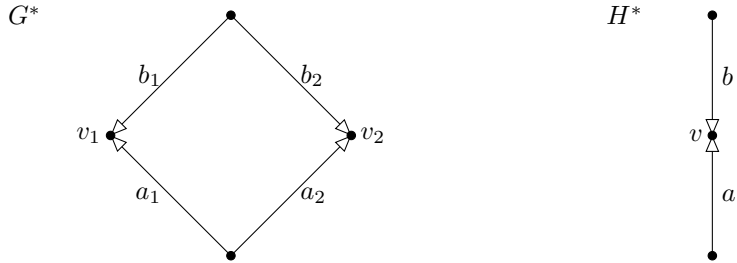
\[
\begin{tikzpicture}[>=open triangle 45,scale=0.8] 

\foreach \x in {(0,0),(-2,2),(2,2),(0,4)} \filldraw \x circle (2pt);
\draw  [->]  (0,0)--(-2,2);
\draw  [->]  (0,0)--(2,2);
\draw  [->]  (0,4)--(-2,2);
\draw  [->]  (0,4)--(2,2);

\node[left] at (-2,2) {$v_1$};
\node[right] at (2,2) {$v_2$};
\node[left] at (-1,1) {$a_1$};
\node[right] at (1,1) {$a_2$};
\node[left] at (-1,3) {$b_1$};
\node[right] at (1,3) {$b_2$};
\node[left] at (-3,4) {$G^*$};

\foreach \x in {(8,0),(8,2),(8,4)} \filldraw \x circle (2pt);

\node[left] at (7,4) {$H^*$};
\draw  [->]  (8,0)--(8,2);
\draw  [->]  (8,4)--(8,2);

\node[left] at (8,2) {$v$};
\node[right] at (8,1) {$a$};
\node[right] at (8,3) {$b$};

\end{tikzpicture}
\]

Recall that by Theorem \ref{thm:deg3} we know that $\ideal(G^*,\theta)$ is generated by its elements of degree at most $3$. Since we assumed that  $\ideal(G^*,\theta)$ is not generated in degree two, it follows from Proposition \ref{prop:equivclasses}  that there is a lattice point $s \in  \semigr(G^*,\theta)_3$ such that $\sim_s$ has more than one equivalence classes. Throughout the proof we will identify the  arrows in $G^*_1 \setminus \{a_1,a_2,b_1,b_2\}$ with the arrows of $H^*_1 \setminus \{a,b\}$, and the vertices in $G^*_0 \setminus \{v_1, v_2\}$ with the vertices in $H^*_0 \setminus \{v\}$. 

We will first deal with the case when $0 \in \{s(a_1), s(a_2), s(b_1), s(b_2)\}$. By symmetry we may assume that $s(a_1) = 0$, which implies that $s(b_1) = 3$. Set $\theta'$ to be the weight of $H^*$ for which $\theta'(v) = 1$, $\theta(b^-) = \theta(b_1^-) + 1$ and $\theta' = \theta$ on the rest of the vertices. 
Define the linear map $\phi:\mr^{G^*_1} \rightarrow \mr^{H^*_1}$, as $\phi(x)(a) = x(a_2)$, $\phi(x)(b) = x(b_2)$ and $\phi(x)(c) = x(c)$ for $c\in G^*_1 \setminus \{a_1,a_2,b_1,b_2\}$. Since for each $m \in \{m \in \semigr(G^*,\theta)_1 \mid\: m \leq s\}$ we have that $m(a_1) = 0$ and $m(b_1) = 1$ and $\phi$ is an order-preserving isomorphism when restricted to the set $\{x \in \mr^{G^*_1} \mid x(a_1) = 0, x(b_1) = 1\}$, it follows that $\phi$ maps the set $\{m \in \semigr(G^*,\theta)_1 \mid\: m \leq s\}$ bijectively onto the set $\{m \in \semigr(H^*,\theta')_1 \mid\: m \leq \phi(s)\}$, moreover that $m_1 + m_2 \leq s$ if and only if $\phi(m_1)+\phi(m_2) \leq \phi(s)$, whence $m_1 \sim_s m_2$ if and only if $\phi(m_1) \sim_{\phi(s)} \phi(m_2)$. Consequently if $\sim_s$ has more than one equivalence classes then so does $\sim_{\phi(s)}$. Since $\phi(s) \in \semigr(H^*,\theta')_3$ we conclude, by  Proposition \ref{prop:equivclasses}, that $\ideal(H^*,\theta')$ is not generated in degree two. 
 
For the remainder of the proof we assume  $0 \not\in \{s(a_1), s(a_2), s(b_1), s(b_2)\}$. After a possible relabeling of the arrows it is sufficient to deal with the following two cases: \newline
Case I: $s(a_1) = 2$, $s(a_2) = 2$, $s(b_1) = 1$, $s(b_2) = 1$ \newline
Case II: $s(a_1) = 2$, $s(a_2) = 1$, $s(b_1) = 1$, $s(b_2) = 2$ \newline
Set $\theta'$ to be the weight of $H^*$ for which $\theta'(v) = 2$ and $\theta' = \theta$ on the rest of the vertices.  Define the linear map $\phi:\mr^{G^*_1} \rightarrow \mr^{H^*_1}$, as $\phi(x)(a) = x(a_1)+x(a_2)$, $\phi(x)(b) = x(b_1)+x(b_2)$. Note that $\phi$ maps the set $\semigr(G^*,\theta)_i$ onto $\semigr(H^*,\theta')_i$ for all $i \in \mathbb{N}$. 
For $m \in \semigr(H^*,\theta')_1$ let us denote by $m[i,j]$ the element of $\phi^{-1}(m)$ with $m[i,j](a_1) = i$ and $m[i,j](a_2) = j$ (if it exists). For $m \in \semigr(H^*,\theta')_1$ if $m(a) = 0$ then $\phi^{-1}(m) \cap \semigr(G^*,\theta)_1= \{m[0,0]\}$, if $m(a) = 2$ then $\phi^{-1}(m) \cap \semigr(G^*,\theta)_1 = \{m[1,1]\}$, moreover if $m(a) = 1$ then $\phi^{-1}(m) \cap \semigr(G^*,\theta)_1= \{m[1,0], m[0,1]\}$. We claim that in both Case I and Case II we have $$\phi^{-1}(\{m \in \semigr(H^*,\theta')_1 \mid\: m \leq \phi(s)\}) \cap  \semigr(G^*,\theta)_1 = \{m \in \semigr(G^*,\theta)_1 \mid\: m \leq s\}.$$ Indeed, the right hand side is contained in the left hand side by $m \in \phi^{-1}(\phi(m))$, and the left hand side is contained in the right hand side since any element of $\semigr(G^*,\theta)_1$ takes value at most one on the arrows $\{a_1, a_2, b_1, b_2$ and  $0 \not\in \{s(a_1), s(a_2), s(b_1), s(b_2)\}$.
{\bf Claim I.} Let $m \in  \semigr(H^*,\theta')_1$ such that $m(a) = 1$ and $m \leq \phi(s)$. Then $m[1,0] \sim_s m[0,1]$.

In Case I, by $m[1,0] \leq s$ and normality of $\polytope(G^*,\theta)$ we have that there exist $n_1, n_2 \in \semigr(G^*,\theta)_1$ such that $$s = m[1,0] + n_1[0,1] + n_2[1,1].$$ Now we have that $m[0,1] \leq m[1,0] + n_1[0,1]$ hence $m[0,1] \sim_s n_2[1,1] \sim_s m[1,0]$. \newline
Similarly in Case II, by $m[0,1] \leq s$ we have that there exist  $n_1, n_2 \in \semigr(G^*,\theta)_1$ such that $$s = m[0,1] + n_1[1,0] + n_2[1,0].$$ Since $m[1,0] \leq m[0,1]+n_1[1,0]$ it follows that $m[1,0] \sim_s n_2[1,0] \sim_s m[0,1]$, proving Claim I.\newline
{\bf Claim II.} Let $m_1, m_2 \in  \semigr(H^*,\theta)_1$ be such that $m_1 + m_2 \leq \phi(s)$. Then there exists $m_1' \in \phi^{-1}(m_1) \cap  \semigr(G^*,\theta)_1$ and $m_2' \in \phi^{-1}(m_2) \cap  \semigr(G^*,\theta)_1$ such that $m_1' + m_2' \leq s$.

In Case I if $m_1(a) = m_2(a) = 1$ then for $m_1' = m_1[1,0]$ and $m_2' = m_2[0,1]$ we have $m_1' + m_2' \leq s$. If $m_1(a)$ and $m_2(a)$ do not both equal $1$ the claim holds for any choice of $m_1' \in \phi^{-1}(m_1) \cap  \semigr(G^*,\theta)_1$ and $m_2' \in \phi^{-1}(m_2) \cap  \semigr(G^*,\theta)_1$.\newline
In Case II if $m_i(a) = 1$ for $i = 1,2$ then set $m_i'(a) = m_i[1,0]$ otherwise choose $m_i'(a)$ to be the unique element in $\phi^{-1}(m_1) \cap  \semigr(G^*,\theta)_1$ to obtain $m_1' + m_2' \leq s$, completing the proof of Claim II.

We are ready to show that if $\sim_s$ has at least two equivalence classes then so does $\sim_{\phi(s)}$. Indeed assume that  $\sim_{\phi(s)}$ has only one equivalence class and choose $m_1, m_2 \in  \{m \in \semigr(G^*,\theta)_1 \mid\: m \leq s\}$. By the definition of $\sim_{\phi(s)}$ there is a chain of elements $n_{1,\dots,k} \in \{m \in \semigr(H^*,\theta')_1 \mid\: m \leq \phi(s)\})$ such that $n_1 = \phi(m_1)$, $n_k = \phi(m_k)$ and $n_i + n_{i+1} \leq \phi(s)$ for $i = 1, \dots, k-1$. Now set $n_1' = m_1$, $n_k' = m_2$ and choose $n_i' \in \phi^{-1}(n_i)$ for $i = 2,\dots,k-1$. By Claim I and Claim II we have that $n_i' \sim_s n_{i+1}'$ for $i = 1,\dots, k-1$, hence by transitivity $m_1 \sim_s m_2$ for an arbitrary choice of $m_1$ and $m_2$ contradicting the assumption that $\sim_s$ has at least two equivalence classes.
\end{proof}

\begin{proposition}\label{prop:multiplearrow_cor}
If for some $d\ge 2$ every at most $(d-1)$-dimensional quiver polytope has a toric ideal generated in degree two, then every $d$ dimensional quiver polytope whose toric ideal is not generated in degree two can be realized as $\polytope(G^*, \theta)$ where $G$ is a simple, prime $3$-regular graph on $2d-2$ vertices. 
\end{proposition}
\begin{proof}
By Lemma \ref{lem:productpoly} if all the lower dimensional quiver polytopes have toric ideals generated in degree two it is sufficient to consider prime quiver polytopes, which by  Theorem \ref{thm:allquivers} can be realized as $\polytope(G^*, \theta)$ where $G^*$ is a prime $3$-regular graph on $2d-2$ vertices. Assume that $G^*$ has multiple arrows running between some vertices. By  Lemma~\ref{lemma:centering} for some $\uk$ the toric ideal $\ideal(\polytope(\quiver,\theta)_{\uk})$ is not generated in degree two. By the assumption that every at most $(d-1)$-dimensinal  quiver polytope has a toric ideal generated in degree two we have that $\dim (\polytope(G^*, \theta)_\uk) = d = \chi(G^*)$. Hence by Theorem~\ref{thm:listcells} there is a weight $\theta'$ on $G^*$ which takes value $1$ on every sink such that $\polytope(\quiver,\theta)_{\uk}$ is integral-affinely equivalent to $\polytope(G^*, \theta')$. Now we can apply Lemma \ref{lem:multiplearrow} to $\polytope(G^*, \theta')$ to obtain a graph $H$ after collapsing some multiple arrows and a weight $\theta''$ such that $\ideal(H^*,\theta'')$ is not generated in degree two. On the other hand we have $d = \chi(G^*) > \chi(H^*) \geq \dim(\polytope(H^*,\theta''))$ contradicting the assumption in the Corollary.
\end{proof}

In the case $d = 3$, there is only one 3-regular simple graph on $4$ vertices, the complete graph $K_4$. For $d = 4$ there are two 3-regular simple graphs on $6$ vertices 
(see \cite{bussemaker_etal}), 
the complete bipartite graph $K_{3,3}$ and the prism graph $Y_3$, shown in the picture below:

$$
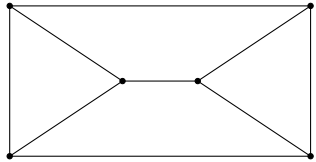
\begin{tikzpicture}[scale=0.5]  
\foreach \x in {(0,0),(0,4),(3,2),(5,2),(8,0),(8,4)} \filldraw \x circle (2pt); 
\draw (0,0)--(0,4)--(3,2)--(0,0); \draw (5,2)--(8,0)--(8,4)--(5,2);\draw (0,0)--(8,0);\draw (3,2)--(5,2); \draw (0,4)--(8,4); 
\end{tikzpicture}$$ 

Next we apply Theorem~\ref{thm:listcells} to list the maximal dimensional cells in each of these cases.

\begin{proposition}\label{prop:listofcells}
\begin{itemize}
\item[(i)] Every $3$-dimensional quiver cell of the form $\polytope(K_4^*,\theta)_\unull$ is integral-affinely equivalent to a $3$-simplex. 
\item[(ii)] There are two different $4$ dimensional quiver cells  $\polytope(Y_3^*,\theta)_\unull$ up to integral-affine equivalence. The toric ideals of both of these are  generated by their elements of degree at most $2$. 
\item[(iii)] There are two different $4$ dimensional quiver cells  $\polytope(K_{3,3}^*,\theta)_\unull$ up to integral-affine equivalence. One of these is a $4$-simplex and the other is the Birkhoff polytope $B_3$. Moreover $\polytope(K_{3,3}^*,\theta)_\unull$ is integral-affinely equivalent to $B_3$ precisely when $\theta$ is $-1$ on the vertices that belong to one of the classes in the bipartition of $K_{3,3}$, $-2$ on the vertices that belong to the other and $1$ on the valency $2$ sinks. 
\end{itemize}
\end{proposition}
\begin{proof}
 By Theorem~\ref{thm:listcells} to list all the maximal dimensional cells $\polytope(G^*,\theta)_\unull$ for a $3$-regular graph $G$ we just have to list all the ways we can pick the position of the weight $-1$ and $-2$ vertices in $G$ up to automorphism. We will provide a figure for each case,  indicating the weights on the sources of $G^*$ and calculate the resulting quiver polytope when $\theta$ takes value $1$ on the valency $2$ sinks (which are not shown on the pictures). In the case of $K_4$ there is only one choice:   

$$
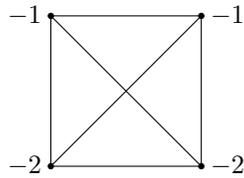
\begin{tikzpicture}[scale=0.5]  
\foreach \x in {(0,0),(0,4),(4,0),(4,4)} \filldraw \x circle (2pt); 
\draw (0,0)--(0,4)--(4,4)--(4,0)--(0,0); \draw (0,0)--(4,4); \draw (4,0)--(0,4); 
\node [left] at (0,0) {$-2$};
\node [left] at (0,4) {$-1$};
\node [right] at (4,4) {$-1$};
\node [right] at (4,0) {$-2$};
\end{tikzpicture}$$ 

It is easy to check that the resulting quiver cell is a $3$-simplex.
In the case of $Y_4$ there are $3$ ways to place the weights up to automorphism:

$$
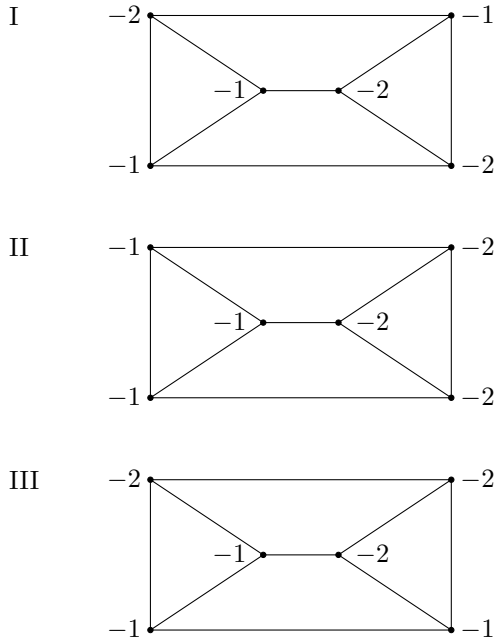
\begin{tikzpicture}[scale=0.5]  
\foreach \x in {(0,0),(0,4),(3,2),(5,2),(8,0),(8,4)} \filldraw \x circle (2pt); 
\draw (0,0)--(0,4)--(3,2)--(0,0); \draw (5,2)--(8,0)--(8,4)--(5,2);\draw (0,0)--(8,0);\draw (3,2)--(5,2); \draw (0,4)--(8,4); 
\node[left] at (0,0) {$-1$};
\node[left] at (0,4) {$-2$};
\node[left] at (2.8,2) {$-1$};
\node[right] at (5.2,2) {$-2$};
\node[right] at (8,0) {$-2$};
\node[right] at (8,4) {$-1$};
\node[right] at (-4,4) {I};
\end{tikzpicture}$$

$$\begin{tikzpicture}[scale=0.5]  
\foreach \x in {(0,0),(0,4),(3,2),(5,2),(8,0),(8,4)} \filldraw \x circle (2pt); 
\draw (0,0)--(0,4)--(3,2)--(0,0); \draw (5,2)--(8,0)--(8,4)--(5,2);\draw (0,0)--(8,0);\draw (3,2)--(5,2); \draw (0,4)--(8,4); 
\node[left] at (0,0) {$-1$};
\node[left] at (0,4) {$-1$};
\node[left] at (2.8,2) {$-1$};
\node[right] at (5.2,2) {$-2$};
\node[right] at (8,0) {$-2$};
\node[right] at (8,4) {$-2$};
\node[right] at (-4,4) {II};
\end{tikzpicture}$$

$$\begin{tikzpicture}[scale=0.5]  
\foreach \x in {(0,0),(0,4),(3,2),(5,2),(8,0),(8,4)} \filldraw \x circle (2pt); 
\draw (0,0)--(0,4)--(3,2)--(0,0); \draw (5,2)--(8,0)--(8,4)--(5,2);\draw (0,0)--(8,0);\draw (3,2)--(5,2); \draw (0,4)--(8,4); 
\node[left] at (0,0) {$-1$};
\node[left] at (0,4) {$-2$};
\node[left] at (2.8,2) {$-1$};
\node[right] at (5.2,2) {$-2$};
\node[right] at (8,0) {$-1$};
\node[right] at (8,4) {$-2$};
\node[right] at (-4,4) {III};
\end{tikzpicture}$$ 

In case I we obtain a quiver cell with $2$ singular and $4$ smooth vertices, and straightforward calculation (we shall give the details below) shows that the corresponding toric ideal is generated by a single quadratic binomial. Similar calculations yield that 
in case II the resulting quiver cell is integral-affinely equivalent to a $2$-dimensional square.  Finally case III is gives a $4$-simplex.

We explain in detail Case I. Applying six times the reduction step of Lemma~\ref{lemma:reduction} below we obtain the quiver-weight pair $(\quiver,\theta)$ shown in the figure below:  

$$
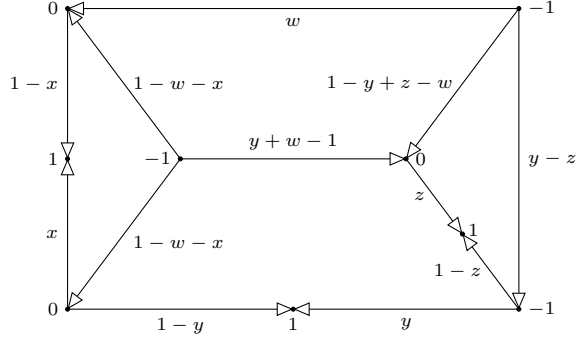
\begin{tikzpicture}[scale=0.25,>=open triangle 45]  
\foreach \x in {(0,0),(0,16),(6,8),(18,8),(24,0),(24,16),(21,4),(12,0),(0,8)} \filldraw \x circle (3pt); 
\draw[->] (24,16)--(24,0); 
\draw[->] (18,8)--(21,4); \draw[->] (24,0)--(21,4); 
\draw[->] (24,16)--(0,16); 
\draw [->] (0,0)--(0,8); \draw [->] (0,16)--(0,8); \draw [->] (0,0)--(12,0); 
\draw [->] (24,0)--(12,0);  \draw [->] (6,8)--(18,8); \draw [->] (6,8)--(0,0); 
 \draw [->] (6,8)--(0,16); \draw [->] (24,16)--(18,8); 
\node[left] at (0,8) {\scriptsize{$1$}}; 
\node[left] at (0,4) {\scriptsize{$x$}}; 
\node[left] at (0,12) {\scriptsize{$1-x$}}; 
\node[left] at (19.5,6) {\scriptsize{$z$}}; \node[right] at (20.8,4.2) {\scriptsize{$1$}}; 
\node[left] at (22.5,2) {\scriptsize{$1-z$}}; 
\node[below] at (12,16) {\scriptsize{$w$}}; 
\node[above] at (12,8) {\scriptsize{$y+w-1$}};
\node[right] at (3,12) {\scriptsize{$1-w-x$}};
\node[left] at (21,12) {\scriptsize{$1-y+z-w$}};
\node[left] at (0,0) {\scriptsize{$0$}};
\node[left] at (0,16) {\scriptsize{$0$}};
\node[left] at (6,8) {\scriptsize{$-1$}};
\node[right] at (18,8) {\scriptsize{$0$}};
\node[right] at (24,0) {\scriptsize{$-1$}};
\node[right] at (24,16) {\scriptsize{$-1$}};
\node[right] at (24,8) {\scriptsize{$y-z$}}; 
\node[below] at (12,0) {\scriptsize{$1$}}; 
\node[below] at (18,0) {\scriptsize{$y$}}; 
\node[below] at (6,0) {\scriptsize{$1-y$}}; 
\node[right] at (3, 3.5) {\scriptsize{$1-w-x$}};
\end{tikzpicture}$$ 
The coordinates $x,y,z,w$ corresponding to the arrows in the complement of a chosen spanning tree of  $\quiver$ can be used as free parameters in the affine solution space 
of the system \eqref{eq:quiverpolytope}. So $\polytope(\quiver,\theta)$ is integral-affinely equivalent to the polytope $\polytope\subset \mr^4$ (with lattice $\mz^4 \subset \mr^4$) 
given by the following inequalities (obtained from the condition that the coordinate corresponding to each arrow is non-negative): 
\[0\le x,y,z,w \le 1; \qquad 
1\le w+y\le 1+z; \qquad 
z\le y \]
The polytope $\polytope$ has six lattice points $b_i$ ($i=1,\dots,6$) whose $x,y,z,w$-coordinates are shown below: 
\[
\begin{array}{c|cccccc}
&  b_1  &  b_2  &  b_3  &  b_4  &  b_5  &  b_6  \\ 
\hline\hline  
 x  &  0  &  0  &  0  &  1  &  1  &  0     \\
 y  &  1  &  1  &  1  &  1  &  1  &  0  \\
 z  &  1  &  0  &  1  &  0  &  1  &  0  \\ 
 w  &  1  &  0  &  0  &  0  &  0  &  1  \\
\hline 
&  1  &  1  &  1  &  1  &  1  &  1  
\end{array} 
\]
The graded semigroup  
$\semigr(\quiver,\theta)$ is generated by $\semigr(\quiver,\theta)_1=\{\tilde b_i\mid i=1,\dots,6\}$ where $\tilde b_i$ is obtained by appending the fifth coordinate $1$ to $b_i$  
The columns of the $5\times 6$ matrix above are $\tilde b_1,\dots,\tilde b_6$, and since this matrix has rank $5$, any integral linear dependency between the $\tilde b_i$ is an integer multiple of 
\[\tilde b_2-\tilde b_3-\tilde b_4+\tilde b_5=0.\] 
It follows that all semigroup relations between the $\tilde b_i$ follow from the relation 
\[\tilde b_2+\tilde b_5=\tilde b_3+\tilde b_4\] 
and consequently $\ideal(\polytope)$ is generated by  the quadratic element 
$t_{b_2}t_{b_5}-t_{b_3}t_{b_4}$.

For $K_{3,3}$ there are two possibilities to place the weights:

$$
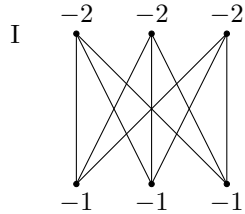
\begin{tikzpicture}[scale=0.5]  
\foreach \x in {(0,0),(2,0),(4,0),(0,4),(2,4),(4,4)} \filldraw \x circle (2pt); 
\draw (0,0)--(0,4)--(2,0)--(2,4)--(4,0)--(4,4); \draw (0,0)--(2,4); \draw (0,0)--(4,4)--(2,0); \draw (0,4)--(4,0);
\node[above] at (0,4) {$-2$};
\node[above] at (2,4) {$-2$};
\node[above] at (4,4) {$-2$};
\node[below] at (0,0) {$-1$};
\node[below] at (2,0) {$-1$};
\node[below] at (4,0) {$-1$};
\node[right] at (-2,4) {I};
\end{tikzpicture}
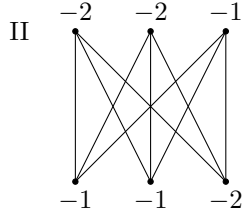
$$ 

$$\begin{tikzpicture}[scale=0.5]  
\foreach \x in {(0,0),(2,0),(4,0),(0,4),(2,4),(4,4)} \filldraw \x circle (2pt); 
\draw (0,0)--(0,4)--(2,0)--(2,4)--(4,0)--(4,4); \draw (0,0)--(2,4); \draw (0,0)--(4,4)--(2,0); \draw (0,4)--(4,0);
\node[above] at (0,4) {$-2$};
\node[above] at (2,4) {$-2$};
\node[above] at (4,4) {$-1$};
\node[below] at (0,0) {$-1$};
\node[below] at (2,0) {$-1$};
\node[below] at (4,0) {$-2$};
\node[right] at (-2,4) {II};
\end{tikzpicture}$$ 

The quiver in case I yields the Birkhoff polytope and the quiver in case II yields a $4$-simplex.
\end{proof}

\begin{lemma}\label{lemma:reduction} 
Suppose that $\quiver$ has a vertex $v$ which is a valency $2$ sink, $a\neq b$ are the arrows with $a^+=v=b^+$, $a^-$ is a source and $\theta$ is a weight with 
$\theta(a^-)=-1$, $\theta(v)=1$. 
Denote by $\quiver'$ the quiver obtained by removing the vertex $v$ and replacing the 
pair $a,b$ of arrows by a single arrow $c$ with $c^-=a^-$ and $c^+=b^-$. 
Moreover, let $\theta'$ be the weight with $\theta'(c^+)=\theta(b^-)+1$ and $\theta'(w)=\theta(w)$ for all vertices $w$ different from $c^+$: 
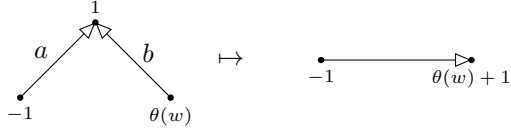
\[\begin{tikzpicture}[scale=1,>=open triangle 45] 
\node [right] at (2.5,0.5) {$\mapsto$}; 
\node [below] at (0,0) {\scriptsize{$-1$}}; \node [below] at (2,0) {\scriptsize{$\theta(w)$}}; 
\foreach \x in {(0,0),(1,1),(2,0),(4,0.5),(6,0.5)} \filldraw \x circle (1pt); 
\node [above] at (1,1) {\scriptsize{$1$}};  \node [above, left] at (0.5,0.6) {$a$}; \node[above,right] at (1.5,0.6) {$b$}; 
\draw [->] (0,0)--(1,1); \draw [<-] (1,1)--(2,0);  \draw [->] (4,0.5)--(6,0.5);
\node [below] at (4,0.5) {\scriptsize{$-1$}}; \node [below] at (6,0.5) {\scriptsize{$\theta(w)+1$}};
\end{tikzpicture} \]
Then $\polytope(\quiver,\theta)$ is integral-affinely equivalent with 
$\polytope(\quiver',\theta')$. 
\end{lemma} 

\begin{proof} Denote by   $H\subset \mr^{\quiver_1}$ the affine solution space of the system \eqref{eq:quiverpolytope} of linear equations. 
Since $a^-$ is a source, $\theta(a^-)=-1$ implies that 
$x(a)\le 1$ for all $x\in H':=\{x\in H\mid x(d)\ge 0 \mbox{ if }d^-=a^-\}$. 
Note that $x(b)=\theta(v)-x(a)=1-x(a)$ for any $x\in H$. Therefore for $x\in H'$ 
we have $x(b)\ge 0$. 
It follows by Lemma 4.4 in \cite{domokos-joo} that the arrow $b$ is \emph{contractible} (in the sense of loc. cit.). Contracting $b$ we get the pair $(Q',\theta')$ and  so 
$\polytope(\quiver,\theta)$ is integral-affinely equivalent to $\polytope(\quiver',\theta')$ by definition of contractibility.  
\end{proof} 

We are ready to prove the main result of this section.

\begin{theorem}\label{thm:cells}
Let $\quiver$ be a quiver without oriented cycles.
\begin{itemize}
\item[(i)] If $dim(\polytope(\quiver,\theta)) \leq 3$ then $\ideal(\quiver,\theta)$ is generated in degree two.
\item[(ii)] If $dim(\polytope(\quiver,\theta)) = 4$ then either  $\ideal(\quiver,\theta)$ is generated in degree two, or $\polytope(\quiver,\theta)$ is integral-affinely equivalent to the Birkhoff polytope $B_3$, when $\ideal(\quiver,\theta)$ is generated by a single cubic element  and some quadratic elements.
\end{itemize}
\end{theorem}
\begin{proof} (i)
It is known that up to integral-affine equivalence the only one-dimensional quiver polytopes are $\polytope(\quiver,(-d,d))$ where $\quiver$ is the Kronecker quiver with two vertices and two arrows from the first vertex to the second (see for example the first paragraph of the proof of Proposition 5.1 in \cite{domokos-joo}); the associated toric ideal is zero when $d=1$ and is generated by the quadratic Veronese relations for $d>1$ (in fact the only 
one-dimensional projective toric variety is $\mathbb{P}^1$).  
In dimension two every projective toric quiver variety is smooth 
(see for example  Proposition 5.1 of \cite{domokos-joo}), so in particular every cell is smooth. Hence by Theorem \ref{thm:0-1deg2} below the toric ideal of any cell is generated by its quadratic elements and then by Lemma \ref{lemma:centering} we see that the toric ideal of any two-dimensional projective toric quiver variety is generated by its quadratic elements. 
Given that, to prove the statement in dimension three by 
Proposion~\ref{prop:multiplearrow_cor} it is sufficient to show that 
$\ideal(\polytope(K_4,\theta))$ is generated in degree two for the only simple 3-regular prime graph $K_4$ on $2\cdot 3-2=4$ vertices. Now this holds by 
Proposition \ref{prop:listofcells} (i), Lemma \ref{lemma:centering} and the two-dimensional case settled above. 

(ii) We turn to the four-dimensional case. Taking into account (i), 
Proposition~\ref{prop:multiplearrow_cor}, Proposition~\ref{prop:listofcells} (ii) and (iii),  and 
Lemma \ref{lemma:centering} it remains to show that if 
$\ideal(\polytope(K_{3,3}^*,\theta))$ is not generated in degree two, then 
$\polytope(K_{3,3}^*,\theta)$ itself is 
integral-affinely equivalent to the Birkhoff polytope $B_3$. 

Let $\{u_1,u_2,u_3\}\cup \{w_1,w_2,w_3\}$ be the vertex set of $K_{3,3}$. 
Denote by $v_{i,j}$ the sink in $K_{3,3}^*$ placed on the edge connecting $u_i$ and $w_j$, 
denote by $a_{i,j}$ (respectively $b_{i,j}$) the arrow of $K_{3,3}^*$ from $u_i$ (respectively 
$w_j$)  to $v_{i,j}$, as it is shown in the figure below.  

$$
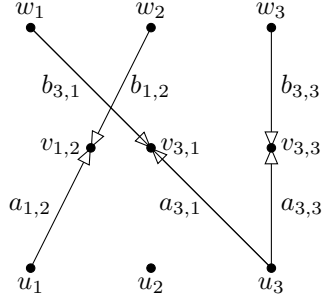
\begin{tikzpicture}[>=open triangle 45,scale=0.8]   
\foreach \x in {(0,0),(2,0),(4,0),(0,4),(2,4),(4,4),(1,2),(4,2),(2,2)} \filldraw \x circle (2pt); 
\draw [->] (0,0)--(1,2); \draw [->] (2,4)--(1,2); \draw (4,0)--(0,4); \draw [->] (4,0)--(4,2); \draw [->] (4,4)--(4,2);
\draw[->] (4,0)--(2,2); \draw [->] (0,4)--(2,2); 
\node[right] at (2,2) {$v_{3,1}$}; \node[left] at (3,1) {$a_{3,1}$}; \node[left] at (1,3) {$b_{3,1}$};
\node[left] at (0.5,1) {$a_{1,2}$}; \node[right] at (1.5,3) {$b_{1,2}$}; 
\node[right] at (4,1) {$a_{3,3}$}; \node[right] at (4,3) {$b_{3,3}$}; 
\node[left] at (1,2) {$v_{1,2}$};
\node[above] at (0,4) {$w_1$};
\node[above] at (2,4) {$w_2$};
\node[above] at (4,4) {$w_3$};
\node[right] at (4,2) {$v_{3,3}$};
\node[below] at (0,0) {$u_1$};
\node[below] at (2,0) {$u_2$};
\node[below] at (4,0) {$u_3$};
\end{tikzpicture}$$ 

By Proposition~\ref{prop:listofcells} (iii) the toric ideal of a cell of $\polytope(K_{3,3}^*,\theta)$ is generated in degree at most two unless the cell is integral-affinely equivalent to the Birkhoff polytope $B_3$. 
Supose that $\polytope(K_{3,3}^*,\theta)_{\uk}$ is integral-affinely equivalent to $B_3$. 
Denote by $\omega$ the weight 
\[\omega(u_i)=-1, \quad \omega(w_i)=-2,\quad \omega(v_{i,j})=1 \qquad (i,j=1,2,3).\] 
By Proposition~\ref{prop:listofcells} (iii) we have  $\polytope(K_{3,3}^*,\theta)_\uk = \uk + \polytope(K_{3,3}^*,\omega)_\unull$. On the other hand  
$\polytope(K_{3,3}^*,\omega)_{\unull}=\polytope(K_{3,3}^*,\omega)$.    
The six vertices of $\polytope(K_{3,3}^*,\omega)_\unull \simeq B_3$ can be indexed by the elements of the symmetric group $S_3$: denote by $\sigma_{pqr}$  the vertex defined by $\sigma_{pqr}(a_{1,p}) = \sigma_{pqr}(a_{2,q}) = \sigma_{pqr}(a_{3,r}) =1$ and $\sigma_{pqr}(a_{i,j}) = 0$ for $(i,j) \notin \{(1,p),(2,q),(3,r)\}$, where $\{p,q,r\}=\{1,2,3\}$ 
(note that a point $x\in\polytope(K_{3,3}^*,\theta)$ is uniquely determined by the coordinates $(x(a_{i,j})\mid i,j=1,2,3)$, since $x(b_{i,j})=\theta(v_{i,j})-x(a_{i,j})$). 
With this notation $\ideal(\polytope(K_{3,3}^*,\omega))$ is generated by a single cubic relation 
corresponding to 
$s = \sigma_{123} + \sigma_{312} + \sigma_{231} = \sigma_{132} + \sigma_{321} + \sigma_{213}$. Therefore 
$\ideal(\polytope(K_{3,3}^*,\theta)_\uk)$ is generated by a single 
cubic relation 
corresponding to $s'=s+3\uk\in\semigr(\polytope(K_{3,3}^*,\theta)_\uk)_3$ and the equality 
\[s' = \sigma'_{123} + \sigma'_{312} +\sigma'_{231} =\sigma'_{132} + \sigma'_{321} + \sigma'_{213}\]
where $\sigma'_{pqr}=\uk+\sigma_{pqr}$.  
In particular, there are two equivalence classes with respect to the relation $\sim_{s'}$ 
on $\semigr(\polytope(K_{3,3}^*,\theta)_{\uk})_1$, namely 
$\{\sigma'_{123}, \sigma'_{312},\sigma'_{231}\}$ and 
$\{\sigma'_{132}, \sigma'_{321}, \sigma'_{213}\}$. 
If $\uk(a_{i,j})>0$ for some $i,j$, say $\uk(a_{1,1}) > 0$, then  
$\sigma'_{123}+\sigma'_{132}=\sigma_{123}+\uk+\sigma_{132}+\uk \leq s'$, 
and hence by Proposition~\ref{prop:equivclasses}  
the element  
\[t_{\sigma'_{123}}t_{\sigma'_{312}}t_{\sigma'_{231}}-
t_{\sigma'_{132}}t_{\sigma'_{321}}t_{\sigma'_{213}}\in\ideal(\polytope(K_{3,3}^*,\theta))\] 
is contained in the ideal generated by the quadratic elements in $\ideal(\polytope(K_{3,3}^*,\theta))$. 

Suppose finally that  $\uk(a_{i,j}) = 0$ for all $i,j \in \{1,2,3\}$.  
Then $\sigma_{123}+\uk\in \polytope(K_{3,3}^*,\theta)$ implies $\theta(u_i)=-1$ 
for $i=1,2,3$, and hence 
\[\{x(a_{i,j})\mid x\in \mz^{(K_{3,3}^*)_1}\cap \polytope(K_{3,3}^*,\theta)\}\subseteq \{0,1\}.\]
Observe that  $\sigma_{pqr}+\uk\in \polytope(K_{3,3}^*,\theta)_\uk$ for all $\sigma_{pqr}$ 
implies that 
\[\{x(a_{i,j})\mid x\in \mz^{(K_{3,3}^*)_1}\cap \polytope(K_{3,3}^*,\theta)\}= \{0,1\}. \] 
Moreover, since $x(b_{i,j})=\theta(v_{i,j})-x(a_{i,j})$, we have 
\[\{x(b_{i,j})\mid x\in \mz^{(K_{3,3}^*)_1}\cap \polytope(K_{3,3}^*,\theta)\}= \{\theta(v_{i,j}),\theta(v_{i,j})-1\}\] 
and 
\[\uk(b_{i,j})=\theta(v_{i,j})-1.\]
It follows that 
\[\polytope(K_{3,3}^*,\theta)=\uk+\polytope(K_{3,3}^*,\omega)_\unull=\polytope(K_{3,3}^*,\theta)_\uk,\] 
thus $\polytope(K_{3,3}^*,\theta)$ itself 
is integral-affinely equivalent to the Birkhoff polytope $B_3$, which indeed has the above cubic relation that is not generated in lower degree. 
\end{proof}

\begin{corollary}\label{cor:bogvad}
B{\o}gvad's conjecture holds for quiver polytopes of dimension at most four; that is, 
the toric ideal of a smooth quiver polytope of dimension at most four is generated in degree two.  
\end{corollary}


\section{Toric ideals of compressed polytopes}\label{sec:relations}

Let $\polytope$ be a lattice polytope with facet presentation $$\polytope = \{x \in \mr^d |\: \<x,a_F\> \geq c_F \mbox{ for all facets F}\},$$ where $a_F$ is the primitive integer normal vector of the facet $F$. The {\it width} of $\polytope$ with respect to the facet F is then defined as $$\max_{x \in \polytope}\{ \<x,a_F\>-c_F\}.$$
We will call a lattice polytope {\it compressed} if it has width one with respect to all of its facets. Note that the quiver cells occurring in Section \ref{sec:quivercells} are compressed by definition. 
In Proposition \ref{prop:compressed} we recall some equivalent characterizations of compressed polytopes (cf. Theorem 2.3 in \cite{haase-etal} or Theorem 2.4 in \cite{sullivant}).

\begin{proposition}\label{prop:compressed}
For a lattice polytope $\polytope$ the following are equivalent:
\begin{itemize}
\item[(i)] $\polytope$ is compressed.
\item[(ii)] $\polytope$ is integral-affinely equivalent to the intersection of the unit cube with an affine subspace.
\item[(iii)] Every weak pulling triangulation of $\polytope$ is unimodular.
\end{itemize}
\end{proposition}

\begin{remark}\label{remark:compressed_quivercell} 
{\rm In view of Lemma~\ref{lemma:finitecells} (i) and Proposition~\ref{prop:compressed} it is clear that the integral-affine equivalence classes of compressed quiver polytopes are exactly the integral-affine equivalence classes of the quiver cells. Note that any compressed polytope realized as the intersection of the unit cube with an affine subspace equals to its own $\unull$ cell. } 
\end{remark} 

The strategy to prove our result in Section~\ref{sec:quivercells} can be summarized as follows: First we investigated the toric ideals of the compressed quiver polytopes, and 
found that essentially the only example up to dimension four that is not generated in degree two is the toric ideal of the Birkhoff polytope $B_3$ all of whose vertices are singular. 
Consequently when such cells appear in a smooth quiver polytope they must have neighboring cells incident to each vertex. We then proceeded to show that the degree three generator of the ideal of $B_3$ can always be generated in lower degree by using lattice points of one of the neighboring cells. We hope that similar strategy can be applied in proving quadratic generation of toric ideals of other classes of latice polytopes. 
As a first step towards this goal we establish some conditions on the arrangement of the smooth and singular vertices that guarantee that the toric ideal of a compressed polytope 
is generated in degree two. 

\begin{remark}\label{rem:bogvad_compressed} {\rm 
It is known that B{\o}gvad's conjecture holds for compressed polytopes. 
Indeed, a compressed polytope is integral-affinely equivalent to a 0-1 polytope by Proposition~\ref{prop:compressed}, and a smooth 0-1 polytope is the product  
of 0-1 simplices by \cite{kaibel-wolf}, hence the toric ideal of a smooth compressed polytope is generated in degree two (say by Lemma~\ref{lem:productpoly}), moreover it possesses a quadratic quadratic Gr\"{o}bner basis (by \cite[Theorem 13]{sullivant2}). }
\end{remark}

For a lattice polytope $\polytope$ and a lattice point $s \in \semigr(\polytope)_k = \{k\polytope\cap \mz^d\}\times\{k\}$ we define the {\it support} of $s$ as, $$\supp(s):= \{F\mbox{ is a facet of }\polytope \mid\:s\notin kF\times \{k\}\}.$$ Note that when $(\quiver,\theta)$ is tight (cf. \cite{domokos-joo}) this is consistent with our earlier definition of $\supp(m)$ for a lattice point $m$ in the quiver polytope $\polytope(\quiver,\theta)$.

\begin{lemma}\label{lemma:support} 
 Let $\polytope$ be a compressed lattice polytope. For $v,w\in \semigr(\polytope)$ we have $\supp(v+w) = \supp(v)\cup \supp(w)$, in particular $\supp(nv)=\supp(v)$ for any $n\in\mathbb{N}$. 
\end{lemma}
\begin{proof} 
We have $v\in\semigr(\polytope)_k$ and $w\in\semigr(\polytope)_l$. Take a facet 
$F=\{x\mid \<x,a_F\>=c_F\}$ of $\polytope$. Since $\<v,a_F\> \geq kc_F$ and $\<w,a_F\>\geq lc_F$, we have that $\<v+w,a_F\>=(k+l)c_F$ if and only if $\<v,a_F\>=kc_F$ and $\<w,a_F\>=lc_F$. By definition of the support this implies that $F \notin \supp(v+w)$ if and only if $F \notin \supp(v)$ and $F \notin \supp(w)$, which is what we needed to show. 
\end{proof} 

Next we collect some well-known elementary properties of compressed polytopes. 
\begin{lemma}\label{lemma:compprops} 
Let $\polytope$ be a compressed lattice polytope. Then we have the following:
\begin{itemize}
    \item[(i)] $\polytope$ is normal.
    \item[(ii)] For $s \in \semigr(\polytope)_{k}$ and $m \in \semigr(\polytope)_{1}$ with $k > 1$, we have that $m \leq s$ if and only if $\supp(m) \subseteq \supp(s)$.
    \item[(iii)] Every lattice point in $\polytope$ is a vertex. In particular if $m_1,m_2 \in \semigr(\polytope)_{1}$, $m_1\neq m_2$,  then $\supp(m_1) \not\subset \supp(m_2)$.
\end{itemize}
\end{lemma}
\begin{proof}
We refer to Proposition 9.3.20 of \cite{lorea-rambau-santos} for the proof of (i). For (ii) let
$$\polytope = \{x \in \mr^d |\: \<x,a_F\> \geq c_F \mbox{ for all facets F}\},$$ be the facet presentation of $\polytope$ where the $a_F$ are primitive integral vectors. Set $s'$ to be the projection of $s-m$ onto the first $d$ coordinates, and note that $m \leq s$ if and only if $s' \in (k-1)\polytope$.  From the definition of compressed polytopes it follows that $\<s',a_F\> < (k-1)c_F$ if and only if $F \notin \supp(s)$ and $F \in \supp(m)$. Hence $s' \in (k-1)\polytope$ if and only if $\supp(m) \subseteq \supp(s)$.
For (iii) if $m \in \polytope \cap \mz^d$ is not a vertex then it can be obtained as a nontrivial convex combination of some of the vertices. It follows that for some $F$ we have $c_F + 1 > \<m,a_F\> > c_F$ contradicting that both $m$ and $a_F$ are integer vectors.
\end{proof}

We say that two vertices of a polytope are neighbours if they lie on the same edge with no intermediate lattice point. We make the following observations:

\begin{lemma}\label{lem:neighbours}
\begin{itemize}
    \item[(i)] Let $\polytope$ be a lattice polytope. The vertices $v_1, v_2 \in \polytope \cap \mz^d$ are neighbouring if and only if there are no other lattice points whose support is a subset of $\supp(v_1) \cup \supp(v_2)$.
    \item[(ii)] Let $\polytope$ be a compressed lattice polytope.  Let $s$ be an element of $\semigr(\polytope)$ and  $v_1,v_2 \leq s$ two vertices of $\polytope$. Then there is a series of vertices $w_1,\dots,w_k \in \polytope$, such that $w_1,\dots,w_k \leq s$, $v_1 = w_1$, $v_2 = w_k$  and $w_i$ is a neighbour of $w_{i+1}$ for all $i:\;1 \leq i \leq k-1$. 
 \end{itemize}
\end{lemma}

\begin{proof} 
First note that the smallest face containing both $v_1$ and $v_2$ is 
\[\cap_{F \notin \supp(v_1) \cup \supp(v_2)} F\] 
by the definition of the support. Now (i) follows from the fact that $v_1, v_2$ are neighbouring if and only if there is no other lattice points on the smallest face of $\polytope$ containing both of them.

For (ii) take a sequence of neighbouring vertices 
\[w_1,\dots,w_k \in \cap_{F \notin \supp(v_1) \cup \supp(v_2)} F\] 
with $v_1 = w_1$ and $v_2 = w_k$. Now it follows from (ii) of Lemma~\ref{lemma:compprops}  that for $i = 1,\dots,k$ we have that $\supp(w_i) \subseteq \supp(v_1) \cup \supp(v_2) \subseteq \supp(s)$ hence $w_i \leq s$ as required.
\end{proof}

 Now we can state the main results of this section:

\begin{theorem}\label{thm:0-1deg2}
Let $\polytope$ be a compressed lattice polytope. If $\polytope$ contains no neighbouring singular vertices then the toric ideal of $\polytope$ is generated by its quadratic elements.
\end{theorem}

\begin{proof}
Recall that we identify the lattice points of $\polytope$ with $\semigr(\polytope)_1$ by slight abuse of the notation. By Proposition \ref{prop:equivclasses} we have to show that for any $k \geq 3$ and $s \in \semigr(\polytope)_k$ the relation $\sim_s$ has precisely one equivalence class. Let us first show that $v_1\sim_s w_1$  when $v_1,w_1 \leq s$ are neighbouring vertices of $\polytope$. By the assumption at least one of them - say $v_1$ - is a smooth vertex. By normality of $\polytope$ there are - not necessarily distinct - vertices $v_2,\dots,v_k,w_2,\dots,w_k$ such that 
\[s = \sum_{i=1}^{k}v_i=\sum_{i=1}^{k}w_i.\] 
By the assumption that $v_1$ and $w_1$ are neighbours one of the ray generators of $\mathrm{Cone}(\polytope - v_1)$ is $w_1-v_1$. Let $u_1-v_1,\dots,u_{dim(\polytope)-1}-v_1$ denote the rest of the ray generators, that together with $w_1-v_1$ form a $\mz$-basis of the lattice $\mz^d \cap \mathrm{Span}(\mathrm{Cone}(\polytope-v_1))$. Consider the ``localized'' equation 
\[\sum_{i=2}^{k} (v_i-v_1)=\sum_{i=1}^{k} (w_i-v_1).\] 
Each $v_i-v_1$ and $w_i-v_1$ decomposes uniquely as a linear combination of the ray generators of  $\mathrm{Cone}(\polytope- v_1)$ with non-negative integer coefficients, so for the two sides to be equal the decomposition of one of the $v_i - v_1$ has to contain $w_1 - v_1$ with a positive coefficient. It follows that for some positive integer $n$ we have an equation 
\[v_i-v_1=n(w_1-v_1)+\sum_{l=1}^j(u_{k_l}-v_1).\]
Rearranging the above equation we get 
\[v_i + (n+j-1)v_1 =nw_1+\sum_{l=1}^j u_{k_l},\] 
implying by Lemma~\ref{lemma:support} that $\supp(w_1) \subseteq \supp(v_i + (n+j-1)v_1 )$ and hence $\supp(w_1) \subseteq \supp(v_i + v_1 )$. By (ii) of Lemma~\ref{lemma:compprops}  we have that  $w_1 \leq v_i + v_1$, and consequently for any $j \notin \{1,i\}$ we have $v_j + w_1 \leq s$ and hence $v_j \sim_s w_1$. Since $v_j \sim_s v_1$ for all $j$, we also have $v_1 \sim_s w_1$. We have established that for any $s$ of degree at least $3$ if $v \leq s$ and $w \leq s$ are neighbouring vertices then we have $v \sim_s w$. Now we are done by Lemma \ref{lem:neighbours}.
\end{proof}

\begin{example}{\rm The presence of any given number of smooth vertices does not guarantee that the toric ideal is generated in degree two, as shown by the following example.
Consider the complete bipartite quiver $\cbq_{3,3}$ and let us write $v_{1,2,3}$ for its sources, $u_{1,2,3}$ for its sinks and $a_{i,j}$ for the arrow pointing from $v_i$ to $u_j$. Let $\quiver$ be a quiver we obtain from $\cbq_{3,3}$ after adding a new vertex $w$, arrows $b_{1,2,3}$ from $v_{1,2,3}$ to $w$, and arrows $c_{1,2,3}$ from $w$ to $u_{1,2,3}$. Set $\theta(v_{1,2,3}) = -1$, $\theta(u_{1,2,3}) = 1$ and $\theta(w) = 0$. In the polytope $\polytope(\quiver,\theta)$ every arrow is incident to a source of weight $-1$ or a sink of weight $1$, hence the coordinates of any lattice point of $\polytope(\quiver,\theta)$ lie in the set $\{0,1\}$. From the definition of quiver polytopes it then follows that $\polytope(\quiver,\theta)$ is the intersection of an affine subspace with the unit cube which by Proposition \ref{prop:compressed} implies that it is compressed. Denoting by $\sigma_{ijk}$ the lattice points of $\polytope(\quiver,\theta)$ that correspond to perfect matchings of $\cbgr_{3,3}$ as in the proof of Theorem \ref{thm:cells}, and setting  $s = \sigma_{123} + \sigma_{312} + \sigma_{231} = \sigma_{132} + \sigma_{312} + \sigma_{213}$, we see that $\sim_s$ has two equivalence classes since $\polytope(\quiver,\theta) \cap \{x \leq s\}$ contains no lattice points other than the $\sigma_{ijk}$. On the other hand consider the vertex $m \in \polytope(\quiver,\theta)$ defined as $m(b_{1,2,3}) = m(c_{1,2,3})  = 1$ and $m(a_{i,j}) = 0$ for all $i,j \in \{1,2,3\}$. The quiver with vertices $\quiver_0$ and arrows $\supp(m)$ is connected, hence $m$ is a smooth vertex by Proposition \ref{prop:smoothvertex}. Now adding multiple copies of the arrows $b_{1,2,3}$ and $c_{1,2,3}$ we can obtain an arbitrarily large amount of smooth vertices in $\polytope(\quiver,\theta)$.
}
\end{example}

\begin{theorem}\label{thm:grobner}
Let $\polytope$ be a compressed lattice polytope. If $\polytope$ has at most one singular vertex then the toric ideal of $\polytope$ has a quadratic Gr\"{o}bner basis.
\end{theorem}
\begin{proof} 
The case when $\polytope$ has no singular vertices is known, 
see Remark~\ref{rem:bogvad_compressed}. 
Let us denote by $v_1,\dots,v_k$ the vertices of $\polytope$, let  $R = \mc[t_{v_1},\dots,t_{v_k}]$ and $\varphi$ denote the surjection $R \twoheadrightarrow \mc[\semigr(\polytope)]$ defined by $\varphi(t_{v_i}) = x^{v_i}$. Note that throughout the proof we will use the symbol $|$ for divisibility in the ring $\mc[\semigr(\polytope)]$. As before we will write the facet inequality corresponding to $F$ as $\<x,a_F\> \geq c_F$. Since $\polytope$ is compressed, for each vertex $w$ of $\polytope$ we have that  $\<x,a_F\> = c_F + 1$ if $F \in \supp(w)$ and  $\<x,a_F\> = c_F $ if $F \notin \supp(w)$.

Let us enumerate the facets of $\polytope$ as $F_1, \dots, F_n$ and order the vertices of $\polytope$ as  $v_j > v_k$ if and only if for the smallest $i$ such that $\<v_j, a_{F_i}\> \neq \<v_k, a_{F_i}\>$ we have  $F_i \in \supp(v_j)$. 
After a possible renumbering of the facets we may assume that if $\polytope$ has a singular vertex $v$, then its support  
consists of the first few facets. 
Since in a compressed polytope the supports of distinct vertices can not contain each other, this assumption implies that $v$ is maximal with respect to $\leq$. (If $\polytope$ has no singular vertices, then the facets can be numbered arbitrarily.) Now we define the monomial ordering $\preceq$ on $R$, such that for two monomials $m_1 = \prod_i (t_{v_i})^{l_i}$ and $m_2 = \prod_i (t_{v_i})^{j_i}$ we have $m_2 \preceq m_1$ if and only if either $\deg(m_1) > \deg(m_2)$ or $\deg(m_1) = \deg(m_2)$ and for the smallest $v_i$ (with respect to the ordering $\leq$) such that $l_i \neq j_i$ we have $j_i > l_i$. 

We claim that the quadratic elements of the ideal of relations $\ker(\varphi)$ are a 
Gr\"{o}bner basis under the ordering $\preceq$.  We call a monomial $T$ {\it initial} if it is the initial monomial with respect to $\preceq$ of some element of $\ker(\varphi)$; that is, there exists an $f\in \ker(\varphi)$ such that $T$ is a monomial of $f$ and for all other monomials $U$ of $f$ we have $U\preceq T$. 
The remaining monomials are called {\it standard}. Record that a monomial $m$ is standard if and only if it is minimal with respect to $\preceq$ amongst the monomials $n$ with $\varphi(m)=\varphi(n)$. 
To prove the claim it is sufficient to show that if $m$ is an initial monomial with $\deg(m)\ge 2$, then 
$m$ is divisible by an initial monomial of a quadratic element of $\ker(\varphi)$. We will prove this by induction on the degree. The case $\deg(m)=2$ is trivial. 

Now assume that 
$m =  \prod_{i=1}^j t_{w_i}$ is an initial monomial, where the $w_i$ are not necessarily distinct vertices of $\polytope$ satisfying $t_{w_1} \succeq t_{w_2} \succeq \dots \succeq t_{w_j}$ and $j \geq 3$. Note that $w_j$ is a smooth vertex, since if it was the unique singular vertex then - since we set it to be maximal with respect to the ordering $\leq$ - we would have $w_1 = \dots = w_j$ and since $\supp(w_j)$ does not contain the support of any other vertex $m = (t_{w_j})^j$ would be the unique monomial that maps to $\varphi(m)$ contradicting that it is initial. 
If $t_{w_j}$ is minimal with respect to $\preceq$ in the set $\{t_{v_i}: x^{v_i}|\varphi(m)\}$ then $ \prod_{i=1}^{j-1} t_{w_i}$ can not be standard, otherwise -- by the definition of $\preceq$ and the characterization of standard monomials mentioned above -- $m$ would also be standard.  Thus $ \prod_{i=1}^{j-1} t_{w_i}$ is an initial monomial, hence  is divisible by a quadratic initial monomial by the induction hypothesis, implying in turn that $m$ is divisible by a quadratic initial monomial. It remains to deal with the case when there is a vertex $v' \in \polytope$ such that $ t_{v'}\prec t_{w_j}$ and $x^{v'}|\varphi(m)$. Let $b$ be the smallest integer such that $\<w_j, a_{F_b}\> > \<v', a_{F_b}\> $. Denoting the neighbouring vertices of $w_j$ by $u_1,\dots,u_{\dim(\polytope)}$, by smoothness of $w_j$ we have that $v' - w_j$ can be uniquely written as  
\begin{equation}\label{eq:alpha} 
v' - w_j=\sum_{i=1}^{\dim(\polytope)} \alpha_i(u_i-w_j)
\end{equation} 
where the coefficients $\alpha_i$ are non-negative integers. Since $\<v' - w_j, a_{F_b}\> =-1$, multiplying both sides of \eqref{eq:alpha} by $a_{F_b}$ we conclude that for some index $l$ we have $\<u_l, a_{F_b}\>  = \<w_j, a_{F_b}\> - 1 = c_{F}$.  Moreover, since $\<u_i, a_{F}\> \in \{c_F, c_F + 1\}$ we have that for any fixed $k$ either all of the $\<u_i-w_j, a_{F_k}\>$ are non-negative or all of them are non-positive. Since $\alpha_i \geq 0$ this also implies that whenever $\<v' - w_j, a_F\> = 0$ for some $F$ we also have $\<u_i - w_j, a_F\> = 0$ for all $i = 1, \dots, \dim(\polytope)$. It follows in particular that $\supp(u_i) \subseteq \supp(w_j)\cup \supp(v')$ for all $i$.

Since $b$ was chosen to be the smallest index where  $\<w_j, a_{F_b}\>$ and $\<v', a_{F_b}\>$ differ, we have that $\<u_l, a_{F_k}\> = \<w_j, a_{F_k}\> = \<v', a_{F_k}\>$ for $k<b$. This in turn  implies that $t_{u_l} \prec t_{w_j}$. By $\supp(u_l) \subseteq \supp(w_j)\cup \supp(v')$ and Lemma~\ref{lemma:compprops} (ii) we conclude that $x^{u_l} | \varphi(m)$, and hence 
\[w_1+\cdots +w_j=u_l+v_1+\cdots+v_{j-1}\] 
for some vertices $v_1,\dots,v_{j-1}\in\polytope$. 
Rewrite the above equality as 
\[\sum_{i=1}^{j-1}(w_i-w_j)=(u_l-w_j)+\sum_{i=1}^{j-1}(v_i-w_j).\] 
For some $r \in \{1,\dots,j-1\}$ the element $u_l-w_j$ has positive coefficient in the expansion of $w_r-w_j$ as a non-negative integral linear combination of the basis elements 
$u_1-w_j,\dots,u_{\dim(\polytope)}-w_j$ in the lattice $\mz^d\cap \mathrm{Span}(\cone(\polytope-w_j))$: 
\[w_r-w_j=\sum_{i=1}^{\dim(\polytope)} p_i(u_i-w_j), \quad p_i\in \mn,\quad p_l>0.\]
This implies that $\supp(u_l)\subseteq \supp(w_r)\cup\supp(w_j)$ by the same argument that was given after \eqref{eq:alpha}. Thus by  Lemma~\ref{lemma:compprops}  (ii) we have 
$x^{u_l} | \varphi(t_{w_j}t_{w_r})$. Now  
$t_{u_l}\prec t_{w_j} \preceq t_{w_r}$,  implying that the quadratic monomial $t_{w_j}t_{w_r}$ is not standard. We are done as $t_{w_j}t_{w_r} | m$.
\end{proof}

A finitely generated graded $k$-algebra $A$ is called a {\it Koszul algebra} if the ground field $k$ has a linear graded free resolution over $A$. By a result of Priddy from \cite{priddy}, the existence of a quadratic Gr\"{o}bner basis of the ideal of relations is a satisfactory condition for the Koszul property of the algebra, so we also have the following corollary:

\begin{corollary}
Let $\polytope$ be a compressed polytope. If $\polytope$ has at most one singular vertex then $\mc[\semigr(\polytope)]$ is Koszul.
\end{corollary}

\begin{example}{\rm
Consider the following quiver $\quiver$:
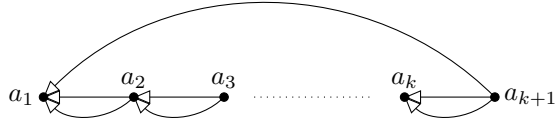
\[
\begin{tikzpicture}[>=open triangle 45,scale=0.8] 

\foreach \x in {(0,0),(1.5,0),(3,0),(6,0),(7.5,0)} \filldraw \x circle (2pt); 
\draw  [->]  (1.5,0) to (0,0);
\draw  [->]  (1.5,0) to [out=225,in=-45] (0,0);
\draw  [->]  (3,0) to  (1.5,0);
\draw  [->]  (3,0) to [out=225,in=-45] (1.5,0);
\draw [dotted] (3.5,0)--(5.5,0);
\draw  [->]  (7.5,0) to (6,0);
\draw  [->]  (7.5,0) to [out=225,in=-45] (6,0); 
\draw  [->]  (7.5,0) to [out=135,in=45] (0,0);

\node[left] at (0,0) {$a_1$};
\node[above] at (1.5,0) {$a_2$};
\node[above] at (3,0) {$a_3$};
\node[above] at (6,0) {$a_k$};
\node[right] at (7.5,0) {$a_{k+1}$};
\end{tikzpicture}
\]

Set $\theta(a_1) = 1$, $\theta(a_{k+1}) = -1$ and $\theta(a_i) = 0$ when $2\leq i \leq k$. $\polytope(\quiver,\theta)$ has a vertex $w$ that takes value $1$ on the arrow from $a_{k+1}$ to $a_1$ and $0$ on the rest of the arrows. Moreover it has $2^k$ vertices that for each $1 \leq i \leq k$ take value 1 on one of the arrows from $a_{i+1}$ to $a_i$ and $0$ on the other, and take value $0$  on the arrow from $a_{k+1}$ to $a_k$. Let us denote these $2^k$ vertices by $v_{i_1,\dots,i_k}$, where $i_j$ equals the value of $v_{i_1,\dots,i_k}$ on the "top" arrow between $a_{i_j+1}$ and $a_{i_j}$. It follows from Proposition \ref{prop:smoothvertex} that all of the  $v_{i_1,\dots,i_k}$ are smooth vertices. It is also easy to verify that $w$ is a singular vertex, for example by checking that it is neighbouring to the $2^k$ vertices  $v_{i_1,\dots,i_k}$ and that the dimension of $\polytope(\quiver,\theta)$ is $k+1$. Furthermore it can be deduced from (ii) of Proposition \ref{prop:compressed} that $\polytope(\quiver,\theta)$ is a compressed lattice polytope. Therefore by Theorem \ref{thm:grobner} the toric ideal $\ideal(\quiver,\theta)$ admits a quadratic Gr\"{o}bner basis and consequently it is  generated by quadratic elements. Indeed one can check that all relations are generated by the ${k \choose 2}$ binomials of the type $t_{v_{0,0,i_3,\dots,i_k}}t_{v_{1,1,i_3,\dots,i_k}}-t_{v_{0,1,i_3,\dots,i_k}}t_{v_{1,0,i_3,\dots,i_k}}$,
}
\end{example}

We conclude this section by giving an example for each positive integer $n$ of a compressed lattice polytope, such that the corresponding ideal of relations is not generated by its elements of degree less than $n$.

Fix $n \geq 2$ and let $\cbgr_{n,n}$ denote the complete bipartite graph on $2n$ vertices, with set of vertices  $(\cbgr_{n,n})_0$ and set of edges  $(\cbgr_{n,n})_1$. Write the bipartition of  $\cbgr_{n,n}$ as $V\coprod W$, where $|V|=|W|=n$  and every $v\in V$ and $w\in W$ is connected by an edge.  Index the coordinates of $\mr^{n^2}$ by the edges of  $\cbgr_{n,n}$. Let $\Pi$ denote the set of the $n!$ perfect matchings of $\cbgr_{n,n}$. Consider the polytope defined as \[\mathcal{P}_n :=  \{x \in \mr^{n^2} \mid\:x\geq 0\quad \forall P \in \Pi:\: \sum_{e\in P} x(e) = 1\}\]
For each $v\in (\cbgr_{n,n})_0$ define the lattice point $m_v$: \[m_v(e) = \begin{cases} 1 &\mbox{if } e\:is\:incident\:to\:v \\ 
0& \mbox otherwise \end{cases}\]
Note that the facets of $\mathcal{P}_n$ can be given as $x(e) = 0$ for edges $e \in (\cbgr_{n,n})_1$, hence for $x \in \mathcal{P}_n$ we can identify $\supp(x)$ with the set $\{e\in(\cbgr_{n,n})_1 | \:x(e) > 0 \}$.
\begin{lemma}\label{lem:knnpolytope}
For any integer $k \geq 1$ and  point $x \in k\mathcal{P}_n$ there is a $v \in (\cbgr_{n,n})_0$ such that $\supp(m_v) \subseteq \supp(x)$.
\end{lemma}

\begin{proof}
First note that for any  $x \in k\mathcal{P}_n$ and cycle of $\cbgr_{n,n}$ with edges $c_1,\dots c_{2i}$, indexed consecutively along the cycle, we have that $x(c_1)+x(c_3)+\dots x(c_{2i-1}) = x(c_2)+x(c_4)+\dots x(c_{2i})$. Otherwise consider a perfect matching $P_1$ of $\cbgr_{n,n}$  containing the edges $c_1,c_3,\dots c_{2i-1}$ and set $P_2$ be the perfect matching we obtain from $P_1$ by replacing $c_j$ by $c_{j+1}$ for each $j < 2i$. For $P_1$ and $P_2$ we have $\sum_{e\in P_1} x(e) \neq \sum_{e\in P_2} x(e)$ contradicting  $x \in k\mathcal{P}_n$.
Now suppose that the conclusion of the  lemma does not hold for some $x$ and let $F :=(\cbgr_{n,n})_1 \setminus \supp(x)$. By the assumption every vertex is incident to at least one edge in $F$. In the subgraph of $\cbgr_{n,n}$ with edge set $F$ take a maximal matching $M$ with edges running between the pairs of vertices $(v_1, w_1), \dots , (v_j, w_j)$, where $v_i\in V$, $w_i\in W$. $M$  can not be a perfect matching, since the coordinates of $x$ sum up to $k$ on every perfect matching. The remaining vertices of $\cbgr_{n,n}$ are  $v_{j+1},\dots v_n\in V$ and  $w_{j+1},\dots w_n\in W$. By the maximality of $M$ the edges running between  $v_{j+1},\dots v_n$ and  $w_{j+1},\dots w_n$ are in $\supp(x)$. Also there has to be some edge of $F$ incident to $v_{j+1}$ and another incident to $w_{j+1}$, let the other endpoint of these two edges be $w^*$ and $v^*$ respectively, and note that $w^* \in \{w_1,\dots,w_j\}$ and $v^* \in\{v_1,\dots,v_j\}$. Now for the length $4$ cycle  $(w^*, v_{j+1}, w_{j+1}, v^*,w^*)$ we have that the first and the third edges do not belong to 
$\supp(x)$, whereas the second edge (running between $v_{j+1}$ and $w_{j+1}$) belongs to $\supp(x)$,  contradicting the observation made at the beginning of the proof.
\end{proof}

\begin {proposition}
\begin{itemize}
\item [(i)] $\mathcal{P}_n$ is a compressed lattice polytope with vertex set $\{m_v \mid v \in K(n,n)_0 \}$.
\item[(ii)] The toric ideal $\ideal(\mathcal{P}_n)$  is not generated by its elements of degree at most $n-1$.
\end{itemize}
\end {proposition}
\begin{proof}
Suppose $x$ is a vertex of $\mathcal{P}_n$. By Lemma \ref{lem:knnpolytope} we have a vertex $v$ with $\supp(m_v) \subseteq \supp(x)$. Set $\lambda := \min\{x(e) \mid e \in \supp(m_v)\}$.  If $x \neq m_v$ then $0 < \lambda < 1$, and we have that $(1-\lambda)^{-1}(x - \lambda m_v) \in \polytope$. Since $x$ is then an interior point of the line segment between $m_v$ and   $(1-\lambda)^{-1}(x - \lambda m_v)$ it can not be a vertex. This shows us that $\mathcal{P}_n$ is indeed a lattice polytope with vertex set $\{m_v \mid v \in V(K(n,n)) \}$ and it follows from the definition and Proposition \ref{prop:compressed} that $\mathcal{P}_n$ is compressed. 

For (ii) set the notation $V=\{v_1,\dots,v_n\}$ and $W=\{w_1,\dots,w_n\}$ for the vertices in the two sides of the bipartition of $K(n,n)$. We have the equality $s = m_{v_1} + \dots + m_{v_n} = m_{w_1} + \dots + m_{w_n} \in n\mathcal{P}_n$. Clearly for any $i,j \in \{1,\dots,n\}$ we have $m_{v_i} + m_{w_j} \not\leq s$, hence $\sim_s$ has two equivalence classes: $m_{v_1},\dots , m_{v_n}$ and   $m_{w_1},\dots , m_{w_n}$. Now we are done by Proposition \ref{prop:equivclasses}.
\end{proof}

\end{document}